\documentclass[a4paper,11pt]{article}

\usepackage{amsmath,amsthm,amssymb}
\usepackage{mathtools}
\usepackage{bigstrut}
\usepackage{bm}
\usepackage{mathrsfs}
\usepackage[shortlabels]{enumitem}
\usepackage{xr}
\theoremstyle{plain}
\newtheorem{theorem}{Theorem}
\newtheorem{lemma}{Lemma}
\newtheorem{corollary}{Corollary}
\newtheorem{proposition}{Proposition}
\theoremstyle{definition}
\newtheorem{definition}{Definition}
\numberwithin{theorem}{section}
\numberwithin{lemma}{section}
\numberwithin{corollary}{section}
\numberwithin{proposition}{section}
\numberwithin{definition}{section}
\numberwithin{equation}{section}
\renewcommand{\vec}[1]{\bm{#1}}
\DeclarePairedDelimiter{\brax}{(}{)}
\DeclarePairedDelimiter{\sqbrax}{[}{]}
\DeclarePairedDelimiter{\setbrax}{\{}{\}}
\DeclarePairedDelimiter{\set}{\{}{\}}
\DeclarePairedDelimiter{\abs}{\lvert}{\rvert}
\DeclarePairedDelimiter{\normDoubleBar}{\lVert}{\rVert}
\newcommand{\clsd}[2]{\sqbrax{#1 , #2}}
\newcommand{\suchthat}{:}
\newcommand{\grad}{\vec{\nabla}}
\newcommand{\bbN}[0]{\mathbb{N}}
\newcommand{\bbP}[0]{\mathbb{P}}
\newcommand{\bbQ}[0]{\mathbb{Q}}
\newcommand{\bbR}[0]{\mathbb{R}}
\newcommand{\bbZ}[0]{\mathbb{Z}}
\newcommand{\calK}[0]{{\mathcal K}}
\newcommand{\frakI}[0]{\mathfrak{I}}
\newcommand{\frakS}[0]{\mathfrak{S}}
\newcommand{\eqdef}{=}
\newcommand{\projdim}{\dim}
\newcommand{\affdim}{\dim}
\newcommand{\cancellation}{{\mathscr C}}
\newcommand{\rank}{\operatorname{rank}}
\newcommand{\sing}{\operatorname{Sing}}
\newcommand{\supnorm}[1]{\normDoubleBar{#1}_\infty}
\newcommand{\supnormBig}[1]{\normDoubleBar[\Big]{#1}_\infty}
\newcommand{\weightbox}{\mathscr{B}}
\newcommand{\singSeries}{\frakS}
\newcommand{\singIntegralBox}{\frakI}
\newcommand{\gradSomethingMultilinear}[2]{\vec{m}^{(#1)} \brax{ #2 }}
\newcommand{\vecsuper}[2]{\vec{#1}^{(#2)}}
\newcommand{\numZeroesInBoxOf}[1]{N_{#1}}
\newcommand{\auxIneqOfCNumSolns}{N^{\operatorname{aux}}_{c}}
\newcommand{\auxIneqOfSomethingNumSolns}[1]{N^{\operatorname{aux}}_{#1}}
\newcommand{\betaDotC}{\vec{\beta}\cdot\vec{c}}
\newcommand{\cubicMatrixAt}[1]{H_{c}\brax{#1}}
\newcommand{\cubicMatrix}{\cubicMatrixAt{\vec{x}}}
\newcommand{\cubicMinorsAt}[2]{\vec{D}^{(c,#1)}\brax{#2} }
\newcommand{\cubicEigenvalueAt}[2]{\lambda_{c,#1}\brax{ #2 }}
\newcommand{\cubicEigenvalue}[1]{\cubicEigenvalueAt{ #1 }{ \vec{x} }}
\newcommand{\cubicMinors}[1]{\cubicMinorsAt{ #1 }{ \vec{x} }}
\newcommand{\cubicMinorsJacobianAt}[2]{J_{\vec{D}^{(c,#1)}}\brax{ #2 }}
\newcommand{\cubicMinorsJacobian}[1]{\cubicMinorsJacobianAt{ #1 }{ \vec{x} }}
\newcommand{\minors}[1]{\vec{D}^{(#1)} }
\newcommand{\minorsMatrix}[2]{#1^{[#2]}}

\begin{document}

				\title{Systems of cubic forms in many variables}
				\author{S. L. Rydin Myerson}
				\maketitle

\begin{abstract}
We consider a system of $R$ cubic forms in $n$ variables, with integer coefficients, which define a smooth complete intersection in projective space. Provided $n\geq 25R$, we prove an asymptotic formula for the number of integer points in an expanding box at which these forms simultaneously vanish. In particular we can handle systems of forms in $O(R)$ variables, previous work having required that $n \gg R^2$.  One conjectures that $n \geq 6R+1$ should be sufficient. We reduce the problem to an upper bound for the number of solutions to a certain auxiliary inequality. To prove this bound we adapt a method of Davenport.
\end{abstract}

		\tableofcontents

		\section{Introduction}

	\subsection{Main result}\label{2.sec:main_result}

Let $c_1(\vec{x}),\dotsc,c_R(\vec{x})$ be homogeneous cubic forms in $n$ variables $x_1,\dotsc,x_n$, with integer coefficients. We treat the simultaneous Diophantine equations
\begin{equation*}
c_1(\vec{x}) = 0,\quad \dotsc, \quad c_R(\vec{x})=0
\end{equation*}
and the corresponding projective variety in $\bbP_\bbQ^{n-1},$ which we call $V(c_1,\dots,c_R) $. We assume throughout that the $c_i$ generate the ideal of $V(c_1,\dots,c_R),$ and are linearly independent. The cubic case of a classic result of Birch gives us:

\begin{theorem}[Birch~\cite{birchFormsManyVars}]\label{2.thm:birch's_result_short}
Let $\weightbox$ be a box in $\bbR^n,$ contained in the box $\clsd{-1}{1}^R,$ and having sides of length at most 1 which are parallel to the coordinate axes. For each $P\geq 1,$ write
\begin{equation*}
\numZeroesInBoxOf{c_1,\dotsc,c_R}(P)
=
\#
\set{ \vec{x}\in\bbZ^n \suchthat
	\vec{x}/P\in\weightbox,\,
	c_1(\vec{x})=0, \dotsc,c_R(\vec{x})=0
}.
\end{equation*}
If the variety $V(c_1,\dots,c_R)$  is a smooth complete intersection, and we have
\begin{equation}
n
\geq
8R^2+9R,
\label{2.eqn:birch's_condition_short}
\end{equation}
then for some  $\singIntegralBox,$ $\singSeries\geq 0$ and  all $P\geq 1,$ the bound
\begin{equation}\label{2.eqn:HL_formula}
\numZeroesInBoxOf{c_1,\dotsc,c_R}(P)
=
\singIntegralBox\singSeries P^{n-3R}
+
O\brax{P^{n-3R-\delta}}
\end{equation}
holds, where the implicit constant depends only on the forms $c_i,$ and $\delta$ is a positive real number depending only on $d$ and $R$. If the variety $V(c_1,\dotsc,c_R)$ has a smooth point over $\bbQ_p$ for each prime $p,$ and a real point whose homogeneous co-ordinates lie in $\weightbox,$ then $\singSeries $ and  $\singIntegralBox$ are positive.
\end{theorem}

In particular this follows from Theorem~1 of Birch~\cite{birchFormsManyVars}, on inserting the bound $\affdim V^\ast \leq R-1$ for the dimension of the variety $V^\ast$ occurring in that result. This bound follows from Lemma~3.1 of Browning and Heath-Brown~\cite{browningHeathBrownDiffDegrees} whenever $V(c_1,\dotsc,c_R)$ is a smooth complete intersection. See \cite[Lemma~\ref{1.lem:nonsing_case}]{quadSystemsManyVars} for details.

We sharpen \eqref{2.eqn:birch's_condition_short} as soon as $R\geq 3$. In \S\ref{2.sec:aux_ineq} we prove:

\begin{theorem}\label{2.thm:main_thm_short}
In Theorem~\ref{2.thm:birch's_result_short} we may replace \eqref{2.eqn:birch's_condition_short} with the condition
\begin{equation}
n
\geq
25R. 
\label{2.eqn:condition_on_n_short}
\end{equation}
\end{theorem}

For example when $R=3$ and $V(c_1,c_2,c_3)$ is a smooth complete intersection, Theorem~\ref{2.thm:main_thm_short} applies when $n \geq 75,$ whereas Birch's theorem requires $n \geq 99$.

The ``square-root cancellation" heuristic suggests that in place of \eqref{2.eqn:birch's_condition_short} the condition $n \geq 6R+1$ should suffice, see for example the discussion around formula (1.5) in Browning~\cite{browningSurveyCircleMethod}. By handling systems of forms in $O(R)$ variables we come within a constant factor of this conjecture.

Our strategy is an extension of our previous work~\cite{quadSystemsManyVars}. In forthcoming papers we further generalise this approach to treat systems of $R$ forms with degree $d \geq 2,$ with rational or real coefficients.


	\subsection{Related work}\label{2.sec:related_work}

We begin with the case when the forms $c_i(\vec{x})$ are diagonal.

In the case of a single diagonal form $c,$ Baker~\cite{bakerDiagCubic7Vars} proves that $V(c)$ has a rational point whenever $n \geq 7$.

Br\"udern and Wooley~\cite{brudernWooleyHPPairsDiagCubics13Vars,
brudernWooleyLBPairsDiagCubics13Vars,brudernWooleyHPDiagCubics} treat diagonal systems in $n\geq 6R+1$ variables. This is the best value of $n$ possible with the classical circle method. In particular they prove the Hasse principle for  $V(c_1,\dots,c_R)$ whenever the  $c_i$ are diagonal, $V(c_1,\dots,c_R)$ is smooth and $n\geq 6R+1$. They also prove an asymptotic formula of the type \eqref{2.eqn:HL_formula} whenever $n \geq 6R+3$ holds, or when $R=2$ and $n \geq 14$ holds~\cite{brudernWooleySimultDiagEqs,brudernWooleyAsympPairsDiagEqs,
brudernWooleyAsymptPairsDiagCubics14Vars}. In the case $R=2$ they prove a Hasse principle for certain pairs of diagonal cubics in as few as 11 variables~\cite{brudernWooleyHPPairsDiagCubics11Vars}.




Returning to the case of general (not necessarily diagonal) forms, we consider the case $R=1$. Let $c$ be a  cubic form. Hooley \cite{hooleyOctonaryCubicsII} proves that if $n=8,$ the variety $V(F)$ is smooth, and the box $\weightbox$ does not contain a point at which the Hessian determinant of $F$ vanishes, then the asymptotic formula \eqref{2.eqn:HL_formula} holds. In this work he assumes a Riemann hypothesis for a certain modified Hasse-Weil $L$-function. When $n=9$ he proves the same result unconditionally, the weaker error term $o(P^{n-3})$ in place of the $O(P^{n-3-\delta})$ from \eqref{2.eqn:HL_formula}. Heath-Brown~\cite{heathBrownCubics14Vars} proves that if $n \geq 14$  then $V(c)$ always has a rational point, regardless of whether it is singular.

In the case $R=2,$ Dietmann and Wooley~\cite{dietmannWooleyPairsCubics} have shown that $V(c_1,c_2)$ always has a rational point when $n \geq 827,$ whether or not it is smooth. 

In the general case $R\geq 1,$ Schmidt~\cite{schmidtSystemsCubics} shows that $V(c_1,\dots,c_R)$ always has a rational point if $n \geq (10R)^{5}$.  Recent work of Dietmann~\cite{dietmannSystemsCubicForms} improves this condition to $n \geq 400~000 R^4$.

	\subsection{Reduction to an auxiliary inequality}\label{2.sec:aux_ineq}
	
	To prove Theorem~\ref{2.thm:main_thm_short} we will use Theorem~\ref{1.thm:manin} from the author's previous work~\cite{quadSystemsManyVars}. This will reduce the problem to proving an upper bound for the number of solutions to the following \emph{auxiliary inequality}.
	
	\begin{definition}\label{2.def:aux_ineq}
		For any $k \geq 1$ and $\vec{t}\in\bbR^k,$ we write  $\supnorm{\vec{t}} = \max_i \abs{t_i}$ for the supremum norm. When $c(\vec{x})$ is a real cubic form in $n$ variables with real coefficients, we define a symmetric matrix 
		\begin{equation}\label{2.eqn:def_of_H}
		H_{c}(\vec{x})
		=
		\frac{1}{\supnorm{ c }} 
		\brax[\bigg]{ \frac{\partial^3 c(\vec{x})}{ \partial x_i \partial x_j }  }_{1 \leq i,j \leq n}
		\end{equation}
		where  $\supnorm{c} = \frac{1}{6} \max_{i,j,k\in\set{1,\dotsc,n}} \abs[\big]{\frac{\partial^3 c(\vec{x})}{\partial x_i x_j x_k}}$. Thus $H_c(\vec{x})$ is the Hessian of the cubic form $c(\vec{x})/\supnorm{ c },$ which has been normalised so that 1 is the absolute value of its largest coefficient. For each  $B \geq 1$  we put $\auxIneqOfSomethingNumSolns{c} \brax{ B }$ for the number of pairs of vectors $(\vec{x}, \vec{y})\in(\bbZ^n)^2$ with
		\begin{equation*}
		\supnorm{\vec{x}},\supnorm{\vec{y}} \leq B, 
		\qquad
		\supnorm{ H_{c}(\vec{x}) \vec{y} } <  B.
		\end{equation*}
	\end{definition}

	We show that this definition of $\auxIneqOfCNumSolns(B)$ above agrees with the one given in \cite[Definition~\ref{1.def:aux_ineq}]{quadSystemsManyVars}. There we consider a degree $d$ polynomial $f$ and a system of multilinear forms $\gradSomethingMultilinear{ f } {\vec{x}^{(1)},\dotsc,\vec{x}^{(d-1)} },$ and when $d=3$ and $f(\vec{x})= c(\vec{x}),$ we see that
	\[
	\gradSomethingMultilinear{ f } {\vec{x}^{(1)},\vec{x}^{(2)} }
	=\supnorm{c} H_c(\vecsuper{x}{1})\vecsuper{x}{2}.
	\]
	It follows that the definition above agrees with Definition~\ref{1.def:aux_ineq} from \cite{quadSystemsManyVars}. The case $d=3$ of Theorem~\ref{1.thm:manin} in~\cite{quadSystemsManyVars} therefore states that:

\begin{theorem}\label{2.thm:manin}
	Let the counting function $\numZeroesInBoxOf{c_1,\dotsc,c_R}(P)$ be as in Theorem~\ref{2.thm:birch's_result_short}. Suppose that for some $C_0 \geq 1$ and $\cancellation > 3R,$ we have
	\begin{equation}\label {2.eqn:aux_ineq_bound_in_manin_thm}
	\auxIneqOfSomethingNumSolns{\betaDotC}(B)
	\leq
	C_0 B^{2n-8\cancellation}
	\end{equation}
	for all $\vec{\beta}\in\bbR^R$ and $B\geq 1,$ where we write $\betaDotC$ for $ \beta_1c_1+\dotsb+\beta_Rc_R$. Then for some  $\singIntegralBox,$ $\singSeries\geq 0$ we have
	\begin{equation*}
	\numZeroesInBoxOf{c_1,\dotsc,c_R}(P)
	=
	\singIntegralBox\singSeries P^{n-dR}
	+
	O\brax{P^{n-dR-\delta}}
	\end{equation*}
	for all $P\geq 1,$ where the implicit constant depends at most on $C_0,$ $\cancellation$ and the $c_i,$ and the positive constant $\delta$ depends at most on $\cancellation,$ $d$ and $R$. The constants $\singIntegralBox$ and $\singSeries $ are positive under the same conditions as in Theorem~\ref{2.thm:birch's_result_short}.
\end{theorem}
	
	We give the following bound for the counting function $\auxIneqOfSomethingNumSolns{c} \brax{ B }$. The proof occupies the bulk of this paper and is completed in \S\ref{2.sec:cubic_aux_ineq_prop}. 
	
\begin{proposition}\label{2.prop:cubic_aux_ineq}
	We call a set $\calK$ of real cubic forms in $n$ variables a \emph{closed cone} if (i) for all $c\in \calK$ and $\lambda\geq 0$ we have $\lambda c\in K,$ and (ii) $\calK$ is closed in the real linear space of cubic forms in $n$ variables.
	
	Let $\calK$ be a closed cone as above, and let $\auxIneqOfSomethingNumSolns{c}(B)$ be as in Definition~\ref{2.def:aux_ineq}. If we set
	\begin{equation}\label{2.eqn:def_of_sigma-sub-calK}
	\sigma_\calK \eqdef 1+\max_{c\in\calK\setminus\set{ 0 }} \projdim \sing V(c),
	\end{equation}
	so that $\sigma_\calK\in\set{0,\dotsc,n-1},$ then for all $\epsilon>0,$ $c\in \calK$ and $B \geq 1$ we have
	\begin{equation}\label{2.eqn:aux_ineq_bound}
	\auxIneqOfSomethingNumSolns{c}\brax{ B }
	\ll_{\calK,\epsilon}
	B^{n+\sigma_\calK+\epsilon}.
	\end{equation}
\end{proposition}

We will outline the proof after deducing Theorem~\ref{2.thm:main_thm_short}.

\begin{proof}[Proof of Theorem~\ref{2.thm:main_thm_short}]
	Suppose that \eqref{2.eqn:condition_on_n_short} holds. We claim that for all $B\geq 1,$ $\epsilon>0$ and $\vec{\beta}\in\bbR^R$ we have
	\begin{equation}\label{2.eqn:aux_ineq_bound_main_thm_proof}
		\auxIneqOfSomethingNumSolns{\betaDotC}(B)
		\ll_{c_1,\dots,c_R,\epsilon}
		B^{n+R-1+\epsilon}
	\end{equation}

where $\betaDotC$ is as in Theorem~\ref{2.thm:manin}. If we set $\cancellation = (n-R+\frac{1}{2})/8$ and let $C_0 $ be sufficiently large in terms of the forms $c_i,$ we can then apply Theorem~\ref{2.thm:manin}. For \eqref{2.eqn:aux_ineq_bound_main_thm_proof} implies \eqref{2.eqn:aux_ineq_bound_in_manin_thm}  on setting  $\epsilon=\frac{1}{2}$ in \eqref{2.eqn:aux_ineq_bound_main_thm_proof}. Moreover we have $\cancellation>3R,$ by \eqref{2.eqn:condition_on_n_short}. So the hypotheses of Theorem~\ref{2.thm:manin} are satisfied, and Theorem~\ref{2.thm:main_thm_short} follows.

Setting $\calK= \set{\betaDotC\suchthat\vec{\beta}\in\bbR^R}$ in Proposition~\ref{2.prop:cubic_aux_ineq}, we see that \eqref{2.eqn:aux_ineq_bound_main_thm_proof} follows from \eqref{2.eqn:aux_ineq_bound} unless
$
\sigma_\calK
>
R-1
$ holds. Suppose for a contradiction that we have 
$
\sigma_\calK
>
R-1
$.

By the definition \eqref{2.eqn:def_of_sigma-sub-calK} there must be $\vec{\beta}\in\bbR^R\setminus\set{\vec{0}}$ with
\begin{equation}\label{2.eqn:very_singular_hypersurface}
\projdim\sing V(\betaDotC)
\geq
R-1.
\end{equation}
We may assume that $
V(c_1,\dots,c_R)
=
V(c_1,\dotsc,c_{R-1},\betaDotC)$ holds, after permuting the $c_i$ if necessary. Since $V(c_1,\dots,c_R)$ is a smooth complete intersection, we then have
\[
V(c_1,\dotsc,c_{R-1})\cap
\sing V(\betaDotC)
\subset
 \sing V(c_1,\dots,c_R),
\]
and so $\projdim \sing V(c_1,\dots,c_R) > \projdim \sing V(\betaDotC) -R$ holds. Thus  \eqref{2.eqn:very_singular_hypersurface} implies that $V(c_1,\dots,c_R)$ is singular, which is false by assumption.
\end{proof}

\subsection{Outline of remaining steps}\label{2.sec:outline}

To prove Proposition~\ref{2.prop:cubic_aux_ineq} we adapt the argument used to prove Lemma~3 in Davenport~\cite{davenportSixteen}, and subsequently a somewhat more general result in \S5 of Schmidt~\cite{schmidtSystemsCubics}. These authors consider the counting function defined by
\begin{equation*}
N^{\operatorname{aux-eq}}_c(B)
\eqdef
\#\set{(\vec{x},\vec{y})\in(\bbZ^n)^2 \suchthat \supnorm{ \vec{x} },\supnorm{ \vec{y} }\leq B,\,
	\cubicMatrix\vec{y}
	= \vec{0}}
\end{equation*}
for some cubic form $c$ with integer coefficients. Davenport proves that either $N^{\operatorname{aux-eq}}_c(B)$ is small, or there is a large rational linear space on which $c$ vanishes. In order to briefly sketch his line of reasoning, we define some additional notation.

\begin{definition}\label{2.def:cubic_eigenvalues}
	Let $\supnorm{ \cubicMatrix }= \max_{i,j} \abs{\cubicMatrix_{ij}} $ and let $ {\cubicEigenvalue{ 1 }} , \dotsc, {\cubicEigenvalue{ n }}$ be the  eigenvalues of the real symmetric matrix $\cubicMatrix,$ listed with multiplicity and in order of decreasing absolute value. Observe that
	\begin{equation}\label{2.eqn:lambda-sub-1_<<_B}
	\abs{\cubicEigenvalue{1}}
	\leq n \supnorm{ \cubicMatrix }
	\leq n^2 \supnorm{\vec{x}}.
	\end{equation}
	For each $i \in \set{ 1,\dotsc,n }$ let $\cubicMinors{ i }$ be the vector of all $i\times i$ minors of $\cubicMatrix,$ arranged in some order. This is a vector of degree $i$ homogeneous forms in the variables $\vec{x},$ with real coefficients. Let $\cubicMinorsJacobian{ i }$  be the Jacobian matrix $\brax{ \partial \Delta^{(c,i)}_j(\vec{x})/\partial x_k }_{jk}$.
\end{definition}

Davenport's argument runs as follows.

\begin{enumerate}[(1)]
	\item\label{2.itm:many_x_with_rank_H_small}
	Let $\sigma\in\set{0,\dotsc,n-1}$. Suppose that  $N^{\operatorname{aux-eq}}_c(B) \gg B^{n+\sigma}$ for some sufficiently large implicit constant. The contribution to this count from any one vector $\vec{x}$ is at most $O(B^{n-\rank \cubicMatrix})$. So there must be an integer $b $ in the set $\set{0,\dotsc,n-1}$ such that at least $\gg B^{\sigma+b}$ integer points $\vec{x}$ satisfy both $\rank\cubicMatrix = b$ and $\supnorm{ \vec{x} }\leq B$.
	
	\item\label{2.itm:count_x_assuming_Delta^k=0_nonsingular}
	If $\sigma, b$ are as in~\ref{2.itm:many_x_with_rank_H_small}, then it follows that there is an integer point $\vecsuper{x}{0}$ such that $\rank\cubicMatrixAt{\vecsuper{x}{0}} = b$ holds and the tangent space to the affine variety $\cubicMinorsAt{ b+1 }{\vec{x}} = \vec{0}$ at the point $\vecsuper{x}{0}$ has dimension $\sigma+b+1$ or more. Equivalently, $\rank\cubicMatrixAt{\vecsuper{x}{0}} = b$ and $\rank \cubicMinorsJacobianAt{ b+1 }{\vecsuper{x}{0}} \leq n-\sigma-b-1$ both hold. This follows from Lemma 2 of Davenport~\cite{davenportSixteen}.
	
	\item\label{2.itm:finding_spaces_X_and_Y}
	If there exists a vector $\vecsuper{x}{0}$ as in~\ref{2.itm:count_x_assuming_Delta^k=0_nonsingular}, then it follows that there exist linear subspaces $X, Y$ of $ \bbQ^n,$ with dimensions $\sigma+b+1$ and $n-b$ respectively, such that for all $\vec{X}\in X$ and $\vec{Y},\vec{Y}' \in Y$ the equality $\vec{Y}^T H_c(\vec{X}) \vec{Y}' = 0$ holds. See Lemma~4 in Schmidt~\cite{schmidtSystemsCubics} or the proof of Lemma~3 in Davenport~\cite{davenportSixteen}.
	
	\item\label{2.itm:contructing_solutions_with_spaces_X_and_Y}
	We conclude that if $N^{\operatorname{aux-eq}}_c(B) \gg B^{n+\sigma}$ then there are spaces $X,Y$ as in~\ref{2.itm:finding_spaces_X_and_Y}. So the space $Z$ defined by $Z= X\cap Y $ is a rational linear space, with dimension at least $\sigma+1,$ such that for all $\vec{Z}\in Z$ the equality $c(\vec{Z}) =0 $ holds.
\end{enumerate}

Our setting differs in three ways from that of Schmidt and Davenport. First, we consider the inequality $\supnorm{ \cubicMatrix\vec{y}} \leq B$ rather than the equation $\cubicMatrix\vec{y} = \vec{0}$. Second, for us the cubic form $c(\vec{x})$ may have real coefficients. And third, rather than concluding that $c(\vec{x})$ has a rational linear space of zeroes, we seek to show that the variety $V(c)$ is very singular.

\subsection{Structure of this paper}\label{2.sec:structure}

In \S\ref{2.sec:many_x_with_rank_H_small} and the three sections \S\S\ref{2.sec:putting_conditions_on_the_Jacobian}-\ref{2.sec:cubic_aux_ineq_prop} we will modify each of the four steps~\ref{2.itm:many_x_with_rank_H_small}-\ref{2.itm:contructing_solutions_with_spaces_X_and_Y} above to accommodate the three changes described at the end of \S\ref{2.sec:outline}. In the remaining section, \S\ref{2.sec:minors_and_eigenvalues}, we prove some technical lemmas relating the minors and eigenvalues of real matrices.

	\subsection{Notation}\label{2.sec:notation}

Throughout, we let  $c,$ $\supnorm{\vec{t}},$ $\supnorm{c},$ $\cubicMatrix$ and $\auxIneqOfSomethingNumSolns{c} \brax{ B }$ be as in Definition~\ref{2.def:aux_ineq}, and we let $\supnorm{\cubicMatrix},$ $\cubicEigenvalue{i},$ $\cubicMinors{i}$ and $\cubicMinorsJacobian{i}$ be as in Definition~\ref{2.def:cubic_eigenvalues}.  We do not require algebraic varieties to be irreducible, and we adopt the convention that $\projdim \emptyset = -1$.  We use Vinogradov's $\ll$ notation and big-$O$ notation in the usual way.

	\section{The eigenvalues of the Hessian matrix $\cubicMatrix$}\label{2.sec:many_x_with_rank_H_small}

We show that if the counting function $\auxIneqOfCNumSolns\brax{ B }$ from Definition~\ref{2.def:aux_ineq} is large, then there are many integer points $\vec{x}$ for which the eigenvalues of $\cubicMatrix$ lie in some fixed dyadic ranges. This corresponds to step~\ref{2.itm:many_x_with_rank_H_small} from \S\ref{2.sec:outline}.

\begin{lemma}\label{2.lem:from_solutions_to_minors}
	Let $H$ be a real symmetric $n\times n$ matrix and let ${\lambda_1} , \dotsc , {\lambda_n}$ be the eigenvalues of the matrix $H,$ listed with multiplicity and in order of decreasing absolute value. Let $C\geq 1$ and $B \geq 1,$ and suppose that $\abs{\lambda_1}\leq C B$ holds. Set
	\begin{equation*}
	N_H(B)
	\eqdef
	\#\set{ \vec{y}\in\bbZ^n \suchthat
		\supnorm{\vec{y}}\leq B,\,
		\supnorm{ H\vec{y} } \leq B }.
	\end{equation*}
	Then we have
	\begin{equation*}
	N_H(B)
	\ll_{C,n}
	\min_{1 \leq i\leq n }
	\frac{B^n}{ 1+ \abs{\lambda_1\dotsm \lambda_i} }.
	\end{equation*}
\end{lemma}

\begin{proof}
	The integral vectors $\vec{y}$ counted by $N_H(B)$ are all contained in the box $\supnorm{\vec{y}} \leq B,$ and in the ellipsoid
	\begin{equation*}
	\set{\vec{t}\in\bbR^n \suchthat \vec{t}^T { H^T H } \vec{t}
		\leq
		n B^2},
	\end{equation*}
	which has principal radii $\abs{\lambda_i}^{-1}\sqrt{n}B$. Hence
	\begin{align*}
	N_H(B)
	&\ll_n
	\prod_{i=1}^n \min\setbrax{ 1+\abs{\lambda_i}^{-1}\sqrt{n}B,\, B}
	\intertext{and as $\abs{\lambda_i}\leq CB$ holds, this is}
	&\leq
	\prod_{i=1}^n \min\setbrax{ 2C\abs{\lambda_i}^{-1}\sqrt{n}B,\, B}.
	\end{align*}
	It follows that
	\begin{equation*}
	N_H(B)
	\ll_{C,n}
	B^n \prod_{i=1}^n \min\setbrax{ \abs{\lambda_i}^{-1}, \, 1}.
	\end{equation*}
	Since the inequalities $\abs{\lambda_1}\geq \dotsb\geq \abs{\lambda_n}$ hold, we deduce that
	\begin{align*}
	N_H(B)
	&\ll_{C,n}
	B^n
	\min\setbrax[\bigg]{ 1, \, \frac{1}{\abs{\lambda_1}}
		,\, \frac{1}{\abs{\lambda_1\lambda_2}}
		,\,\dotsc
		,\, \frac{1}{\abs{\lambda_1\dotsm\lambda_n}} }
	\\
	&\ll
	\min_{1 \leq i\leq n }
	\frac{B^n}{ 1+ \abs{\lambda_1\dotsm \lambda_i} }.\qedhere
	\end{align*}
\end{proof}

\begin{definition}\label{2.def:the_sets_K-sub-k}
Suppose that  $k\in \set{0,\dotsc,n}$ and that $E_1,\dotsc,E_{k+1}\in\bbR$ such that the inequalities $E_1\geq \dotsb \geq E_{k+1} \geq 1$ hold. Then we define $
K_k(E_1,\dotsc,E_{k+1})$ to be the set of all vectors $\vec{x}$ in $\bbR^n$ satisfying the following conditions: the inequality $\supnorm{\vec{x}}\leq B$ holds, and we have
\[
\frac{1}{2} E_i
<
\abs{ \cubicEigenvalue{ i } }
\leq E_i
\]
whenever $1 \leq i \leq k$ holds, and we have
\[
\abs{ \cubicEigenvalue{ i } }
\leq
E_{k+1}
\]
whenever $k+1 \leq i \leq n$ holds.
\end{definition}

\begin{corollary}\label{2.cor:many_x_with_prescribed_eigenvalues} Let $\auxIneqOfCNumSolns\brax{ B }$  be as in Definition \ref{2.def:aux_ineq}, let $\cubicEigenvalue{i}$ and $\cubicMinors{ i }$ be as in Definition~\ref{2.def:cubic_eigenvalues}, and let $K_k(E_1,\dotsc,E_{k+1})$ be as in Definition~\ref{2.def:the_sets_K-sub-k}. For any $B\geq 1,$ one of the following alternatives holds. Either
	\begin{equation}\label{2.eqn:many_x_with_H_small}
	\frac{\auxIneqOfCNumSolns\brax{B}}{B^n \brax{\log B}^n} \ll_n \#\setbrax{\bbZ^n \cap K_0(1)},
	\end{equation}
	or there is $k\in \set{1,\dotsc,n-1}$ and there are $e_1, \dotsc, e_k\in\bbN$ such that the inequalities $\log B \gg_n e_1 \geq \dotsb \geq e_k $ hold and
	\begin{equation}\label{2.eqn:many_x_with_prescribed_eigenvalues}
	\frac{ 2^{e_1+\dotsb+e_k} \auxIneqOfCNumSolns\brax{B}}{B^n \brax{\log B}^n}
	\ll_n
	\#\setbrax[\big]{ \bbZ^n \cap K_k\brax{2^{e_1}, \dotsc, 2^{e_k}, 1} },
	\end{equation}
	or there are $e_1, \dotsc, e_n\in\bbN$ satisfying $\log B \gg_n e_1 \geq \dotsb \geq e_n$  and
	\begin{equation}\label{2.eqn:many_x_with_all_eigenvalues_prescribed}
	\frac{ 2^{e_1+\dotsb+e_n} \auxIneqOfCNumSolns\brax{B}}{B^n \brax{\log B}^n}
	\ll_n
	\#\setbrax[\big]{ \bbZ^n \cap K_{n-1}\brax{2^{e_1}, \dotsc, 2^{e_n}} }.
	\end{equation}
\end{corollary}

\begin{proof}
	Note that in the case that $k=n,$ there are no values of $i$ satisfying $k+1 \leq i \leq n,$ so the  last condition in the definition of $K_k\brax{E_1,\dotsc,E_{k+1}}$ is vacuously true and can be omitted. In particular, if $k=n$ then \eqref{2.eqn:many_x_with_all_eigenvalues_prescribed} follows from \eqref{2.eqn:many_x_with_prescribed_eigenvalues}, because
	\[
	K_n\brax{2^{e_1},\dotsc,2^{e_n},1} \subset K_{n-1}\brax{2^{e_1},\dotsc,2^{e_n}}.
	\]
	So it is enough to prove that either \eqref{2.eqn:many_x_with_H_small} holds or there exist integers $k$ and $e_1,\dotsc,e_k,$ satisfying the inequalities $1\leq k \leq n$ and  $\log B \gg_n e_1 \geq \dotsb \geq e_n,$ such that \eqref{2.eqn:many_x_with_prescribed_eigenvalues} holds.
	
	Now the set $K_0(1),$ together with the sets $K_k\brax{2^{e_1},\dotsc,2^{e_k},1},$ partition the box $\supnorm{\vec{x}}\leq B$ into disjoint pieces. So, if we let
		\begin{equation*}
	N_{\cubicMatrix}(B)
	\eqdef
	\#\set{ \vec{y}\in\bbZ^n \suchthat
		\supnorm{\vec{y}}\leq B,\,
		\supnorm{ N_{\cubicMatrix}(B)\vec{y} } \leq B },
	\end{equation*}
	then we have
	\begin{equation}\label{2.eqn:partition_of_aux_ineq_solns}
	\auxIneqOfCNumSolns\brax{B}
	=
	\sum_{\substack{ \vec{x}\in\bbZ^n \\ \vec{x}\in K_0(1) }} N_{\cubicMatrix}(B)
	+
	\sum_{\substack{ 1 \leq k \leq n \\ e_1 \geq \dotsb \geq e_k \geq 1 \\ e_1 \ll_n \log B }}
	\sum_{\substack{ \vec{x}\in\bbZ^n \\ \vec{x}\in K_k\brax{2^{e_1},\dotsc,2^{e_k},1} }}
	N_{\cubicMatrix} (B).
	\end{equation}
	The total number of terms on the right-hand side of \eqref{2.eqn:partition_of_aux_ineq_solns} is $O_n(\brax{\log B}^n)$ at most, so it follows that either
	\begin{equation}\label{2.eqn:triv_partition_large}
	\sum_{\substack{ \vec{x}\in\bbZ^n \\ \vec{x}\in K_0(1) }}
	N_{\cubicMatrix} (B)
	\gg_n
	\frac{\auxIneqOfCNumSolns\brax{B}}{\brax{\log B}^n}
	\end{equation}
	holds, or else there are $1 \leq k \leq n$ and $e_1 \geq \dotsb \geq e_k \geq 1 $ such that
	\begin{equation}\label{2.eqn:one_partition_large}
	\sum_{\substack{ \vec{x}\in\bbZ^n \\ \vec{x}\in K_k\brax{2^{e_1},\dotsc,2^{e_k},1} }}
	N_{\cubicMatrix} (B)
	\gg_n
	\frac{\auxIneqOfCNumSolns\brax{B}}{\brax{\log B}^n}.
	\end{equation}
	If \eqref{2.eqn:triv_partition_large} holds then the trivial bound $N_{\cubicMatrix} (B) \ll_n B^n$ implies \eqref{2.eqn:many_x_with_H_small}. Suppose instead that \eqref{2.eqn:one_partition_large} holds.
	
	By \eqref{2.eqn:lambda-sub-1_<<_B}, for each real vector $\vec{x}$ the bound $\abs{\cubicEigenvalue{1}} \ll_n B$ holds. So we may apply Lemma~\ref{2.lem:from_solutions_to_minors} with the choice $H = \cubicMatrix$ and some $C$ depending on $n$ only. This shows that
	\begin{equation*}
	N_{\cubicMatrix}(B)
	\ll_{n}
	\frac{B^n}{2^{e_1+\dotsb+e_k}}.
	\end{equation*}
	Substituting this into  \eqref{2.eqn:one_partition_large} we see that \eqref{2.eqn:many_x_with_prescribed_eigenvalues} holds, as claimed.
\end{proof}

%
%

\section{Intermission: Eigenvalues and minors}\label{2.sec:minors_and_eigenvalues}

Here we collect some elementary facts about  the eigenvalues and minors of real matrices which will be needed in \S\S\ref{2.sec:putting_conditions_on_the_Jacobian}-\ref{2.sec:finding_spaces_X_and_Y}.




\begin{lemma}\label{2.lem:matrices_of_minors}
	For each $k,\ell\in \bbN,$ let
	\[
	T_{k,\ell}
	=\set{\vec{a}\in \bbN^k \suchthat 1 \leq a_1 <\dotsb<a_k\leq \ell}.
	\]
	This set has ${\ell \choose k}$ members. For each $k,\ell,m \in \bbN$ such that $k \leq \min\setbrax{\ell,m},$ and each $\ell \times m$ real matrix $L,$ define an ${\ell\choose k}\times {m\choose k}$ real matrix $\minorsMatrix{L}{k}$ by
	\[
	\minorsMatrix{L}{k} = \brax{\minorsMatrix{L}{k}_{\vec{a}\vec{b}}}_{\vec{a}\in T_{k,\ell},\vec{b}\in T_{k,m}},
	\qquad
	\minorsMatrix{L}{k}_{\vec{a}\vec{b}} = \det \brax{\brax{L_{a_i b_j}}_{1 \leq i, j \leq k} },
	\]
	so that the $\minorsMatrix{L}{k}_{\vec{a}\vec{b}}$ are the $k \times k$ minors of $L$. For all $\ell \times m$ matrices $L,$ all $m\times n$ matrices $M$ and all $k \leq \min\setbrax{\ell,m,n}$ we have $\minorsMatrix{(LM)}{k}=\minorsMatrix{L}{k}\minorsMatrix{M}{k}$. That is, we have
	\begin{equation}\label{2.eqn:matrix_multiplication_and_minors}
	\minorsMatrix{(LM)}{k}_{\vec{a}\vec{b}}
	=
	\sum_{\vec{w}\in T_{k,m}}
	\minorsMatrix{L}{k}_{\vec{a}\vec{w}}
	\minorsMatrix{M}{k}_{\vec{w}\vec{b}}.
	\end{equation}
\end{lemma}

\begin{proof}
	Let $\vecsuper{e}{1},\dotsc,\vecsuper{e}{m}$ be the standard basis of $\bbR^m$. Fix $L,\vec{a},\vec{b}$; then each side of \eqref{2.eqn:matrix_multiplication_and_minors} is an alternating multilinear form in those $k$ columns of $M$ whose indices appear in the vector $\vec{b}$. This is some $k$-tuple of $m$-vectors.
	
	If one is given the value of an alternating multilinear form at the $k$-tuple $\vecsuper{e}{z_1},\dotsc, \vecsuper{e}{z_k}$ for each $\vec{z} \in T_{k,m},$ one can extend by linearity and the alternating property to find its value at any $k$-tuple of $m$-vectors. In other words, it suffices to check \eqref{2.eqn:matrix_multiplication_and_minors} when, for some $\vec{z} \in T_{k,m},$ the  $k\times k$ submatrix $(M_{z_i b_j})_{1\leq i,j\leq k}$ is the identity and all other entries of $M$ are zero. In this case both sides of \eqref{2.eqn:matrix_multiplication_and_minors} are equal to $\minorsMatrix{M}{k}_{\vec{z}\vec{b}}$.
\end{proof}

\begin{lemma}\label{2.lem:minors_and_eigenvalues}
	Let $M$ be a real $m\times n$ matrix. Recall that $M^TM$ is positive semidefinite and symmetric. Let the eigenvalues of $M^TM$ be $\Lambda_1^2, \dotsb, \Lambda_n^2$ in decreasing order, where the $\Lambda_i$ are nonnegative and in decreasing order. That is, the $\Lambda_i$ are the \emph{singular values} of $M,$ listed in decreasing order.
	
	In particular, if $M$ is a symmetric matrix, then the $\Lambda_i$ are exactly the absolute values of the eigenvalues of $M,$ by diagonalisation.
	
	Given a natural number $k$ with $k \leq\min\brax{ m,n },$ let $\minors{k}$ be the vector of $k\times k$ minors of $M,$ arranged in some order. Then we have:
	
	\begin{enumerate}[(i)]
		\item\label{2.itm:minors_and_eigenvalues}
		The maximum norm $\supnorm{ \vecsuper{ \Delta }{ k } }$ satisfies	
		\begin{equation}
		\label{2.eqn:minors_and_eigenvalues}
		\supnorm{ \vecsuper{ \Delta }{ k } }
		\asymp_{m,n}
		{\Lambda_1\dotsm \Lambda_k}.
		\end{equation}
		\item\label{2.itm:space_where_matrix_is_big}
		There is a $k$-dimensional linear space $V\subset \bbR^n$ such that for all $\vec{v}\in V$,
		\begin{equation}\label{2.eqn:space_where_matrix_is_big}
		\supnorm{M \vec{v}}
		\gg_{m,n}
		\supnorm{\vec{v}} {\Lambda_k}.
		\end{equation}
		We may take $V$ to be a span of $k$ standard basis vectors $\vecsuper{e}{i}$ in $\bbR^n$.
		\item\label{2.itm:space_where_matrix_is_small}
		For any $C\geq 1,$ either there is an $(n-k+1)$-dimensional linear subspace $X$ of $\bbR^n$ such that
		\begin{align}\label{2.eqn:lambda-sub-k_<<_C^-1}
		\supnorm{M \vec{X}}
		&\leq
		C^{-1}\supnorm{\vec{X}}
		&\text{for all }\vec{X}\in X,
		\end{align}
		or there is a $k$-dimensional linear subspace $V $ of $\bbR^n,$ spanned by standard basis vectors of $\bbR^n,$ such that
		\begin{align*}
		\supnorm{M \vec{v}}
		&\gg_{m,n}
		C^{-1}\supnorm{\vec{v}}
		&\text{for all }\vec{v}\in V.
		\end{align*}
		
	
	\end{enumerate}
\end{lemma}

\begin{proof}[Proof. Part~\ref{2.itm:minors_and_eigenvalues}] \let\qedsymbol\relax
	First we prove the result on the assumption that $M^TM$ is diagonal. Let  the sets $T_{k,\ell}$ and the matrices $\minorsMatrix{L}{k}$  be as in Lemma~\ref{2.lem:matrices_of_minors}. Since $M^TM$ is diagonal with diagonal entries $\Lambda_i^2,$ we have
	\begin{align}
	\sum_{\vec{a} \in T_{k,n}} \Lambda_{a_1}^2\dotsm \Lambda_{a_k}^2
	&=
	\sum_{\vec{a} \in T_{k,n}}
	\minorsMatrix{\brax{M^TM}}{k}_{\vec{a}\vec{a}}
	\nonumber
	\\
	&=
	\sum_{\substack{ \vec{a}\in T_{k,n} \\ \vec{w} \in T_{k,m} }} \brax{\minorsMatrix{M}{k}_{\vec{w}\vec{a}}}^2,
	\label{2.eqn:minors_of_MTM}
	\end{align}
	by \eqref{2.eqn:matrix_multiplication_and_minors}. The left-hand side of \eqref{2.eqn:minors_of_MTM} is $\asymp_n \Lambda_1^2\dotsm \Lambda_k^2,$ and the right-hand side is $\asymp_{m,n} \supnorm{\vec{\Delta}^{(k)}}^2,$ so this proves \eqref{2.eqn:minors_and_eigenvalues}.
	
	Let $O$ be an $n \times n$ orthogonal matrix such that $O^TM^TMO$ is diagonal. Let $\vec{\widetilde{\Delta}}^{(k)}$ be the vector of $k\times k$ minors of $MO$. We claim that the norms $\supnorm{\vec{\widetilde{\Delta}}^{(k)}}$ and $\supnorm{\vec{\Delta}^{(k)}}$ are of comparable size.
	
	Lemma~\ref{2.lem:matrices_of_minors} shows that $\minorsMatrix{\brax{MO}}{k} = \minorsMatrix{M}{k}\minorsMatrix{O}{k},$ and since $\minorsMatrix{\brax{O^T}}{k}\minorsMatrix{O}{k} = \minorsMatrix{I}{k}$ and  $\minorsMatrix{\brax{O^T}}{k}_{\alpha\beta} = \minorsMatrix{O}{k}_{\beta\alpha}$ we see that $\minorsMatrix{O}{k}$ is orthogonal. Hence the maximum norm of the entries satisfies
	\[
	\supnorm{\vec{\widetilde{\Delta}}^{(k)}}
	=
	\supnorm{\minorsMatrix{(MO)}{k}}
	\asymp_{m,n}
	\supnorm{\minorsMatrix{M}{k}}
	=
	\supnorm{\vec{\Delta}^{(k)}}.
	\]
	So in proving \eqref{2.eqn:minors_and_eigenvalues} we may assume that $M^TM$ is diagonal. The result follows.
\end{proof}\begin{proof}[Part~\ref{2.itm:space_where_matrix_is_big}]\let\qedsymbol\relax
	By permuting the rows and columns of $M,$ we may assume that
	\begin{equation}\label{2.eqn:leading_minor_is_biggest}
	\supnorm{\vec{\Delta}^{(k)}}= \abs{\det \brax{M_{ij}}_{1\leq i,j\leq k}}.
	\end{equation}
	Let $\vec{v}$ be in the span of the first $k$ basis vectors. If \eqref{2.eqn:space_where_matrix_is_big} holds for all such $\vec{v}$ then we have proved the lemma.  Since $v_i = 0$ for $i>k,$ one finds that
	\begin{equation*}
	\left(
	\begin{array}{@{}c@{}}
	\begin{matrix}
	M_{11}
	&
	\cdots
	&
	M_{1n}
	\\
	\vdots
	&
	\ddots
	&
	\vdots
	\\
	M_{k1}
	&
	\cdots
	&
	M_{kn}
	\end{matrix}
	\\
	\begin{array}{@{\hspace{.2em}}c|c@{\hspace{.2em}}}
	\hline
	0_{(n-k)\times k}
	&
	I_{n-k}\bigstrut
	\end{array}
	\end{array}
	\right)
	\left(
	\begin{array}{@{}c@{\hspace{\arraycolsep}}c@{}}
	\vec{v}
	&
	\begin{array}{@{}|c@{}}
	0_{1\times(n-1)}
	\\
	\hline
	I_{n-1}
	\end{array}
	\end{array}
	\right)	
	=
	\left(
	\begin{array}{@{}c@{}}
	\begin{matrix}
	\brax{ M\vec{v} }_1
	&
	M_{12}
	&
	\cdots
	&
	M_{1n}
	\\
	\vdots
	&
	\vdots
	&
	\ddots
	&
	\vdots
	\\
	\brax{ M\vec{v} }_k
	&
	M_{k2}
	&
	\cdots
	&
	M_{kn}
	\end{matrix}
	\\
	\begin{array}{@{\hspace{\arraycolsep}\hspace{.8em}}c@{\hspace{\arraycolsep}\hspace{.8em}}|@{\hspace{\arraycolsep}\hspace{.8em}}c@{\hspace{\arraycolsep}\hspace{.8em}}}
	\hline
	0_{(n-k)\times k}
	&
	I_{n-k}
	\bigstrut
	\end{array}
	\end{array}
	\right)
	\end{equation*}
	where we have divided each matrix into three blocks, and $0_{p\times q}$ stands for a $p\times q$ block of zeroes. By \eqref{2.eqn:leading_minor_is_biggest} we have $\pm v_1\cdot\supnorm{ \vec{\Delta}^{(k)}} $ as the determinant of the left-hand side. Expanding the determinant of the right-hand side in the first column, we find that it is equal to
	\begin{align*}
	\pm\supnorm{ \vec{\Delta}^{(k)}} v_1
	&=
	\sum_{\ell=1}^k
	(-1)^{\ell+1}
	\brax{ M\vec{v} }_\ell
	\det
	\brax[\big]{ \brax[\big]{ M_{ij} }_{\substack{ i = 1,\dotsc,k; \,i \neq \ell \\ j = 2,\dotsc,k }} }
	\\
	&\ll
	k
	\supnorm{ M\vec{v} } \supnorm{ \vec{\Delta}^{(k-1)} }.
	\end{align*}
	Note the $(k-1)\times(k-1)$ determinant in which $i$ runs over $1,\dotsc, k$ with the value $\ell$ omitted, and $j$ runs over $2,\dotsc,k$.
	
	By part~\ref{2.itm:minors_and_eigenvalues}, this implies that $\Lambda_k v_1 \ll_{m,n} \supnorm{ M\vec{v} },$ so provided that $ \supnorm{ \vec{v} } =\abs{v_1} ,$ then \eqref{2.eqn:space_where_matrix_is_big} holds.
	
	If we apply the same permutation both to the $v_i$ and to the first $k$ rows of $M,$ then both sides of our claim \eqref{2.eqn:space_where_matrix_is_big} and our assumption \eqref{2.eqn:leading_minor_is_biggest} remain the same. By applying such a permutation we may assume $ \supnorm{ \vec{v} } =\abs{v_1} ,$ and so we have proved \eqref{2.eqn:space_where_matrix_is_big}.
\end{proof}\begin{proof}[Part \protect\ref{2.itm:space_where_matrix_is_small}]
	Let $X$ be the span of the $\Lambda_i^2$-eigenvectors of $M^TM,$ where $i$ runs from $k$ up to $n$. As the matrix $M^TM$ is symmetric, we have $\vec{X}^T M^TM\vec{X}\ll_n \supnorm{\vec{X}}^2 \Lambda_k^2$ for all $\vec{X}\in X,$ and so
	\begin{align*}
		\supnorm{M \vec{X}}
		\ll_{m,n}
		\supnorm{\vec{X}}\Lambda_k
	\end{align*}
	for all $\vec{X}\in X.$ Therefore either this space $X$ satisfies \eqref{2.eqn:lambda-sub-k_<<_C^-1}, or the bound $\Lambda_k\gg_{m,n} C^{-1}$ holds and the existence of the space $V$ follows by part~\ref{2.itm:space_where_matrix_is_big}.
\end{proof}

\section{Counting points in the sets $K_k\brax{E_1,\dotsc,E_{k+1}}$}\label{2.sec:putting_conditions_on_the_Jacobian}

In this section our goal is to estimate the number of integer points in the sets $K_k\brax{E_1,\dotsc,E_{k+1}}$ from Definition~\ref{2.def:the_sets_K-sub-k}. We give the following result.

\begin{lemma}\label{2.lem:union_of_cubes}
	Let $c$ and $\cubicMatrix$ be as in Definition~\ref{2.def:aux_ineq}, let $\cubicEigenvalue{ i },$ $\cubicMinors{ i },$ $\cubicMinorsJacobian{ i }$ be as in Definition~\ref{2.def:cubic_eigenvalues}, and let $K_k\brax{E_1,\dotsc,E_{k+1}}$ be as in Definition~\ref{2.def:the_sets_K-sub-k}. Suppose that $B,C \geq 1,$  $\sigma \in \set{0,\dotsc, n-1},$ and $k\in\set{0,\dotsc,n-\sigma-1},$ and that $CB\geq E_1\geq E_{k+1}\geq 1$. Then at least one of the following holds:
	\begin{enumerate}[label=(\Roman*)$_{k}$,ref=(\Roman*)]
		\item\label{2.itm:union_of_cubes}
		The set $K_k(E_1,\dotsc,E_{k+1})$ may be covered by a collection of at most
		\[
		O_{C,n}\brax{B^{\sigma} (E_1 \dotsm E_{k+1}) E_{k+1}^{-\sigma-k-1} }
		\]
		boxes in $\bbR^n$ of side $E_{k+1}$. Such a box contains $O_n(E_{k+1}^n)$ integral points, so it follows that
		\begin{equation}\label{2.eqn:integer_points_in_E-cubes}
		\#\setbrax{\bbZ^n\cap K_{k}\brax{E_1,\dotsc,E_{k+1}}}
		\ll_{C,n}
		B^{\sigma} (E_1 \dotsm E_{k+1}) E_{k+1}^{n-\sigma-k-1}.
		\end{equation}
		
		\item\label{2.itm:some_Jacobian_close_to_having_small_rank}
		There exist an integer $1 \leq b \leq k,$ a point $\vecsuper{x}{0}\in K_{b}\brax{E_1,\dotsb,E_{b+1}},$ and a $(\sigma+b+1)$-dimensional linear subspace $X$ of $ \bbR^n$ such that
		\begin{align}
		\nonumber
		E_{b+1}
		&<
		C^{-1}E_{b},
		\\
		\label{2.eqn:Jacobian_close_to_having_small_rank}
		\supnorm{\cubicMinorsJacobianAt{ b+1 }{ \vecsuper{x}{0} } \vec{X}}
		&\leq
		C^{-1}
		\supnorm{ \cubicMinorsAt{ b }{ \vecsuper{x}{0} } }\supnorm{ \vec{X} }
		&\text{for all } \vec{X}\in X.
		\end{align}
	\end{enumerate}\begin{enumerate}[resume,label=(\Roman*)]
		\item\label{2.itm:c_close_to_a_cone}
		There is a $(\sigma+1)$-dimensional linear subspace $X$ of $ \bbR^n $ such that
		\begin{align}
		\label{2.eqn:c_close_to_a_cone_at_many_points}
		\supnorm{\cubicMatrixAt{ \vec{X} }}
		&\leq
		C^{-1}
		\supnorm{ \vec{X} }
		&\text{for all } \vec{X}\in X,
		\end{align}
		with $\supnorm{ \cubicMatrixAt{ \vec{X} } }$ as in \S\ref{2.sec:outline}.
	\end{enumerate}
	We have subscripted the first two items to emphasize their dependence on $k$; note that item~\ref{2.itm:c_close_to_a_cone} has no such dependence.
\end{lemma}

In Corollary~\ref{2.cor:applying_davenport's_trick} below, we will use Lemma~\ref{2.lem:union_of_cubes} to bound the quantities \eqref{2.eqn:many_x_with_H_small} and \eqref{2.eqn:many_x_with_prescribed_eigenvalues} from Corollary~\ref{2.cor:many_x_with_prescribed_eigenvalues}. Before proving the lemma, we give a comparison with  step~\ref{2.itm:count_x_assuming_Delta^k=0_nonsingular} in \S\ref{2.sec:outline}. 

If there are many integer points $\vec{x}$ for which $\rank \cubicMatrixAt{ \vec{x} } = b$ holds, then step~\ref{2.itm:count_x_assuming_Delta^k=0_nonsingular} gives us a point $\vecsuper{x}{0}$ for which the matrix $\cubicMinorsJacobianAt{ b+1 }{\vecsuper{x}{0}}$ has a kernel of dimension $(\sigma+b+1)$ or more and $\rank \cubicMatrixAt{ \vecsuper{x}{0} } = b$ holds.

If there are many integer points $\vec{x}$ for which $\vec{x}\in K_k(E_1,\dotsc,E_{k+1}),$ then \eqref{2.eqn:integer_points_in_E-cubes} is false and so either~\ref{2.itm:some_Jacobian_close_to_having_small_rank}$_{k}$ or~\ref{2.itm:c_close_to_a_cone} must hold. Of these, case~\ref{2.itm:some_Jacobian_close_to_having_small_rank}$_{k}$ gives us a point $\vecsuper{x}{0}$ such that $\cubicMinorsJacobianAt{ b+1 }{\vecsuper{x}{0}}$ is small on a $(\sigma+b+1)$-dimensional space. Moreover it states that $\vecsuper{x}{0}\in K_{b}\brax{E_1,\dotsb,E_{b+1}}$ and that $E_{b+1}<C^{-1}E_b,$ so that the $(b+1)$th eigenvalue of the matrix $\cubicMatrixAt{\vecsuper{x}{0}}$ is about $C$ times smaller than the $b$th eigenvalue. Thus~\ref{2.itm:some_Jacobian_close_to_having_small_rank}$_{k}$ gives us a point $\vecsuper{x}{0}$ for which in some sense $\cubicMinorsJacobianAt{ b+1 }{\vecsuper{x}{0}}$  is close to having a kernel of dimension at least $(\sigma+b+1)$ and $\cubicMatrixAt{\vecsuper{x}{0}}$ is close to having rank $b$.

The third case~\ref{2.itm:c_close_to_a_cone} is less directly comparable to step~\ref{2.itm:count_x_assuming_Delta^k=0_nonsingular}. We suggest that it could correspond to the case $b=0$ of  step~\ref{2.itm:count_x_assuming_Delta^k=0_nonsingular}.

\begin{proof}[Proof of Lemma \ref{2.lem:union_of_cubes}]\let\qedsymbol\relax
	The proof is by induction on $k$. Let $c,$ $C,$ $B,$ and $\sigma$ be fixed.
\end{proof}\begin{proof}[The case $k=0$]\let\qedsymbol\relax
	Let $k=0,$  let $CB \geq E_1 \geq 1$ and suppose that alternative~\ref{2.itm:c_close_to_a_cone} does not hold. We claim that alternative~\ref{2.itm:union_of_cubes}$_{0}$ holds, that is $K_0(E_1)$ is covered by $O_{C,n}(B^{\sigma} /E_1^{\sigma})$ boxes of side $E_1$.
	
	As~\ref{2.itm:c_close_to_a_cone}$_{0}$ is false, applying Lemma~\ref{2.lem:minors_and_eigenvalues}\ref{2.itm:space_where_matrix_is_big} to the matrix of the linear map $\vec{x}\mapsto\cubicMatrix$ shows that there is an $(n-\sigma)$-dimensional subspace  $V$ of $ \bbR^n$ with
	\begin{align}\label{2.eqn:space_where_cubic_matrix_large}
	\supnorm{ \cubicMatrixAt{ \vec{v} } }
	&\gg_n
	C^{-1}\supnorm{ \vec{v} }
	&\text{for all }\vec{v}\in V.
	\end{align}

	
	For each $\vec{z}\in\bbR^n,$ let $A_0(\vec{z})$ be the box in $\bbR^n$ defined by
	\[
	A_0(\vec{z})
	=
	\set{\vec{z}+\vec{u}+\vec{v} \suchthat \vec{u}\in V^\perp, \vec{v}\in V,\supnorm{\vec{u}}\leq E_1, \supnorm{\vec{v}}\leq B}.
	\]
	Now $K_0(E_1)$ is contained in the box $\supnorm{\vec{x}}\leq B$. It follows that we can cover $K_0(E_1)$ with a collection of $O_{C,n}(B^{\sigma}/E_1^{\sigma})$ boxes of the form $A_0(\vec{z}),$ each one of which is centred at a point $\vec{z}$ belonging to $K_0(E_1)$.	We will show below that for each $\vec{z}\in K_0(E_1),$ the intersection $A_0(\vec{z})\cap K_0(E_1)$ is contained in a box of side $O_{C,n}(E_1)$. It follows that $K_0(E_1)$ is covered by $O_{C,n}(B^{\sigma}/E_1^{\sigma})$ boxes of side $E_1,$ as claimed.
	
	It remains to let $\vec{z}\in K_0(E_1)$ and let $\vec{y}\in A_0(\vec{z})\cap K_0(E_1),$ and to deduce that $\supnorm{\vec{y}-\vec{z}}\ll_{C,n}E_1$ must hold.
	
	By definition of $K_0(E_1)$ we have $\abs{\cubicEigenvalueAt{1}{\vec{y}}}\leq E_1$ and $\abs{\cubicEigenvalueAt{1}{\vec{z}}} \leq E_1,$ and the bounds $\supnorm{ \cubicMatrixAt{ \vec{y} } }\ll_n E_1$  and  $\supnorm{ \cubicMatrixAt{ \vec{z} } } \ll_n E_1$ follow by \eqref{2.eqn:lambda-sub-1_<<_B}. So we have
	\begin{equation}
	\label{2.eqn:covering_the_set_of_small_H}
	\supnorm{\cubicMatrixAt{ \vec{y} - \vec{z} }} \ll_n E_1.
	\end{equation}	
	Let $\vec{u}\in V^\perp$ and let $\vec{v}\in V$ such that $
	\vec{y}
	=\vec{z}+ \vec{u}+\vec{v}$ holds. Since $\vec{y}$ lies in $A_0(\vec{z}),$ we have $\supnorm{\vec{u}}\leq E_1,$ and with \eqref{2.eqn:covering_the_set_of_small_H} this implies that
	\[
	\supnorm{ \cubicMatrixAt{ \vec{v} } } \ll_n E_1.
	\]
	By \eqref{2.eqn:space_where_cubic_matrix_large} it follows that $\supnorm{ \vec{v} } \ll_n C E_1,$ and hence that $\supnorm{\vec{y}-\vec{z}}\ll_{C,n}E_1,$ as claimed.
\end{proof}\begin{proof}[The inductive step] Let $k\geq 1$ and let $CB \geq E_1\geq \dotsb E_{k+1}\geq 1$. We suppose that~\ref{2.itm:some_Jacobian_close_to_having_small_rank}$_{k}$ and~\ref{2.itm:c_close_to_a_cone} are both false, and claim that~\ref{2.itm:union_of_cubes}$_{k}$ holds.
 
 By induction, at least one of~\ref{2.itm:union_of_cubes}$_{k-1}$,~\ref{2.itm:some_Jacobian_close_to_having_small_rank}$_{k-1}$, or~\ref{2.itm:c_close_to_a_cone}  holds.  Note that of these~\ref{2.itm:c_close_to_a_cone} is false by assumption, and~\ref{2.itm:some_Jacobian_close_to_having_small_rank}$_{k-1}$ is false since it implies~\ref{2.itm:some_Jacobian_close_to_having_small_rank}$_{k}$,  and so~\ref{2.itm:union_of_cubes}$_{k-1}$ must hold.
 
	 Suppose for the time being that
	 \begin{equation}\label{2.eqn:case_E-sub-k+1_small}
	 E_{k+1} < C^{-1}E_k.
	 \end{equation}
	 The contrary case is almost trivial and will be dealt with at the end of the proof. We claim that
		\begin{equation}\label{2.eqn:covering_K-sub-k}
			K_k(E_1,\dotsc,E_{k+1})
			=\bigcup_V
			K_k^{(C,V)}\brax{ E_1,\dotsc,E_{k+1}},
		\end{equation}
		where $V$ runs over those $(n-\sigma-k)$-dimensional subspaces of $\bbR^n$ which are spanned by standard basis vectors, and we define
		\begin{multline}
		K_k^{(C,V)}\brax{ E_1,\dotsc,E_{k+1}}
		\eqdef
		\\
		\big\{ 
		\vec{x}\in K_k\brax{ E_1,\dotsc,E_{k+1}}
		\suchthat
		\supnorm{\cubicMinorsJacobian{k+1} \vec{v}}
		\geq
		C^{-1}
		\supnorm{ \cubicMinors{ k } }\supnorm{ \vec{v} }
		\\
		\text{ for all }
		\vec{v} \in V
		\big\}.
		\label{2.eqn:set_of_x_where_J_is_far_from_having_small_rank}
		\end{multline}	
		We have assumed that $E_{k+1} < C^{-1}E_k$ and  that the case $b=k$ of~\ref{2.itm:some_Jacobian_close_to_having_small_rank}$_{k}$ is false. So the case $b=k$ of \eqref{2.eqn:Jacobian_close_to_having_small_rank} must be false for every $(\sigma+b+1)$-dimensional subspace $X$ of $\bbR^n$ and every 
		$\vecsuper{x}{0}\in K_k(E_1,\dotsc,E_{k+1})$.
	
		That is, for any $\vecsuper{x}{0}\in K_k(E_1,\dotsc,E_{k+1})$ and  any $(\sigma+k+1)$-dimensional linear subspace $X$ of $\bbR^n,$ there is some $	\vec{X}\in X$ such that
		\begin{equation*}
		\supnorm{\cubicMinorsJacobianAt{k+1}{ \vecsuper{x}{0} } \vec{X}}
		>
		C^{-1}
		\supnorm{ \cubicMinorsAt{ k }{ \vecsuper{x}{0} } }\supnorm{ \vec{X} }.
		\end{equation*}
		Applying Lemma~\ref{2.lem:minors_and_eigenvalues}\ref{2.itm:space_where_matrix_is_big} with the choice $M = \cubicMinorsJacobianAt{ k+1 }{ \vecsuper{x}{0} }$ shows that for each $\vecsuper{x}{0}\in K_k(E_1,\dotsc,E_{k+1})$ there is an $(n-\sigma-k)$-dimensional subspace $V$ of $\bbR^n,$ spanned by standard basis vectors, such that
		\begin{equation}\label{2.eqn:space_where_matrix_is_big_application}
		\supnorm{\cubicMinorsJacobianAt{k+1}{ \vecsuper{x}{0} } \vec{v}}
		\geq
		C^{-1}
		\supnorm{ \cubicMinorsAt{ k }{ \vecsuper{x}{0} } }\supnorm{ \vec{v} }
		\end{equation}
		for all $\vec{v}\in V$. This proves \eqref{2.eqn:covering_K-sub-k}, and so to prove~\ref{2.itm:union_of_cubes}$_{k}$ it now suffices to show that for each $(n-\sigma-k)$-dimensional space $V,$ the set  \eqref{2.eqn:set_of_x_where_J_is_far_from_having_small_rank} is covered by a union of $O_{C,n}\brax{B^{\sigma}(E_1\dotsm E_{k+1})E_{k+1}^{-\sigma-k-1}}$ boxes of side $E_{k+1}$.
		
		Let $\epsilon>0$ be a sufficiently small constant depending at most on $C$ and $n,$ and for each $\vec{z}\in\bbR^n$ set
		\begin{equation}\label{2.eqn:def_of_A-sub-k}
		A_k(\vec{z})
		\eqdef
		\set{\vec{z}+\vec{u}+\vec{v} \suchthat \vec{u}\in V^\perp, \vec{v}\in V,\supnorm{\vec{u}}\leq E_{k+1}, \supnorm{\vec{v}}\leq \epsilon E_k}.
		\end{equation}
		We assumed at the start of this inductive step that~\ref{2.itm:union_of_cubes}$_{k-1}$ holds, and so $K_{k-1}\brax{E_1,\dotsc,E_k}$ is covered by a collection of $O_{C,n}(B^{\sigma}(E_1\dotsm E_k)E_k^{-\sigma-k})$ boxes of side $E_k$. Subdivide each of these boxes into $O_{C,n}\brax{E_{k}^{\sigma+k}/E_{k+1}^{\sigma+k}}$ sub-boxes of the form $A_k(\vec{z})$. Since		
		\[
		K_k^{(C,V)}\brax{ E_1,\dotsc,E_{k+1}}\subset K_{k-1}\brax{E_1,\dotsc,E_k},
		\]
		it follows that the set $K_k^{(C,V)}\brax{ E_1,\dotsc,E_{k+1}}$ may be covered by a collection of $O_{C,n}\brax{B^{\sigma}(E_1\dotsm E_{k+1})E_{k+1}^{-\sigma-k-1}}$ boxes of the form $A_k(\vec{z}),$ each of which is centred at a point $\vec{z}$ belonging to the set $K_k^{(C,V)}\brax{ E_1,\dotsc,E_{k+1}}$.
		
		We will show below that for each such box $A_k(\vec{z}),$ the intersection $A_k(\vec{z})\cap K_k^{(C,V)}\brax{ E_1,\dotsc,E_{k+1}}$ is covered by a box of side $O_{C,n}(E_{k+1})$. It follows that each set \eqref{2.eqn:set_of_x_where_J_is_far_from_having_small_rank} is covered by $O_{C,n}\brax{B^{\sigma}(E_1\dotsm E_{k+1})E_{k+1}^{-\sigma-k-1}}$ boxes of side $E_{k+1},$ and by the comments after \eqref{2.eqn:space_where_matrix_is_big_application} this proves the lemma.
		
		We suppose that $\vec{z}\in K_k^{(C,V)}\brax{ E_1,\dotsc,E_{k+1}}$ and that $\vec{y}\in A_k(\vec{z})\cap K_k^{(C,V)}\brax{ E_1,\dotsc,E_{k+1}},$ and we claim that $\supnorm{\vec{y}-\vec{z}}\ll_{C,n} E_{k+1}$ holds. Let $\vec{u}\in V^\perp$ and let $\vec{v}\in V$ such that
		$
		\vec{y}
		=
		\vec{z}
		+\vec{u}
		+\vec{v}
		,$ and note that since $\vec{y}\in A_k(\vec{z}),$ we have
		\begin{equation}\label{2.eqn:bounds_on_u_and_v}
		\supnorm{\vec{u}}\leq E_{k+1},
		\qquad
		\supnorm{\vec{v}}\leq \epsilon E_k.
		\end{equation}
		
		Now the $j$th partial derivatives of the $(k+1)\times(k+1)$ minors $\cubicMinors{ k+1 }$ are linear combinations of the $(k+1-j)\times(k+1-j)$ minors $\cubicMinors{ k+1-j }$ with coefficients of size at most $O_n(1)$. So we have
		\begin{equation*}
		\supnormBig{\frac{\partial^{j} \cubicMinors{k+1} }{ \partial x_{i_1} \dotsm \partial x_{i_j} }}
		\ll_n
		\supnorm{\cubicMinors{k+1-j}},
		\end{equation*}
		and Taylor expansion shows that
		\begin{multline*}
	\cubicMinorsAt{ k+1 }{ \vec{z}+\vec{u}+\vec{v} }
			-\cubicMinorsAt{ k+1 }{\vec{z}}
		\\
		=\cubicMinorsJacobianAt{k+1}{\vec{z}} . \brax{\vec{u}+\vec{v}}
		\nonumber
		+O_n\brax[\Big]{
		\supnorm{ \vec{u}+\vec{v}}^2 \supnorm{\cubicMinorsAt{k-1}{\vec{z}}}}
		\\
		+\dotsb
		+O_n\brax[\Big]{
		\supnorm{ \vec{u}+\vec{v} }^{k} \supnorm{\cubicMinorsAt{1}{\vec{z}}}
		+\supnorm{\vec{u}+\vec{v}}^{k+1}}.
		\end{multline*}
	It follows that
		\begin{multline}
			 \supnorm{\cubicMinorsJacobianAt{k+1}{\vec{z}} \vec{v}}
			 \ll_n{}
			 \supnorm{\cubicMinorsAt{ k+1 }{ \vec{y} }}
			 +\supnorm{\cubicMinorsAt{ k+1 }{\vec{z}}}
			 \\
			 +
				\supnorm{ \vec{u} } \supnorm{\cubicMinorsAt{k}{\vec{z}}}
				+\dotsb
				+
				\supnorm{ \vec{u} }^{k} \supnorm{\cubicMinorsAt{1}{\vec{z}}}
				+
				\supnorm{\vec{u}}^{k+1}
			\\
			+
				\supnorm{ \vec{v} }^2 \supnorm{\cubicMinorsAt{k-1}{\vec{z}}}
				+\dotsb
				+
				\supnorm{ \vec{v} }^{k} \supnorm{\cubicMinorsAt{1}{\vec{z}}}
				+
				\supnorm{\vec{v}}^{k+1}.
						\label{2.eqn:taylor_expansion_of_minors}
		\end{multline}
		
		Since $\vec{y},\vec{z}\in K_k\brax{ E_1,\dotsc,E_{k+1}},$ Lemma~\ref{2.lem:minors_and_eigenvalues} gives us the bounds
		\begin{align}
		\supnorm{\cubicMinorsAt{j}{\vec{z}}}
		&\asymp_n
		\prod_{i=1}^{j} E_i,
		\label{2.eqn:minors_and_E-sub-i}
		&
		\supnorm{\cubicMinorsAt{k+1}{\vec{y}}}
		&\asymp_n
		\prod_{i=1}^{k+1} E_i,
		\end{align}
		and since $\vec{z}\in K_k^{(C,V)}\brax{ E_1,\dotsc,E_{k+1}}$ it follows from \eqref{2.eqn:set_of_x_where_J_is_far_from_having_small_rank} that
		\begin{align}\label{2.eqn:Jacobian_and_E-sub-i}
		\supnorm{
			\cubicMinorsJacobianAt{ k+1 }{ \vec{z} }
			\vec{v}
		}
		&\gg_{n}
		C^{-1} \supnorm{ \vec{v} }
		\prod_{i=1}^k E_i.
		\end{align}
		Substituting \eqref{2.eqn:minors_and_E-sub-i}-\eqref{2.eqn:Jacobian_and_E-sub-i} into \eqref{2.eqn:taylor_expansion_of_minors} yields
		\begin{align*}
		C^{-1} \supnorm{ \vec{v} }
		\ll_n{}&
		\prod_{i=1}^{k+1}E_i+ \supnorm{\vec{v}}^2 \prod_{i=1}^{k-1}E_i + \dotsb +
			\supnorm{\vec{v}}^k E_1
			+
			\supnorm{\vec{v}}^{k+1} 
		\nonumber
		\\
		&+
		\supnorm{\vec{u}} \prod_{i=1}^{k}E_i + \dotsb +
			\supnorm{\vec{u}}^k E_1
			+\supnorm{\vec{u}}^{k+1}.
		\end{align*}
		Applying the bounds from \eqref{2.eqn:bounds_on_u_and_v} and  the inequalities $E_1 \geq \dotsb \geq  E_{k+1},$ we deduce that
		\begin{equation*}
		C^{-1}\supnorm{ \vec{v} }
		\ll_n
		\prod_{i=1}^{k+1} E_i+
		\epsilon\supnorm{ \vec{v} } \prod_{i=1}^{k} E_i.
		\end{equation*}
		Since $\epsilon$ is assumed to be small in terms of $C$ and $n,$ it follows that  $\supnorm{ \vec{v} }\ll_n CE_{k+1}$ holds and hence that $\supnorm{\vec{y}-\vec{z}}\ll_{C,n} E_{k+1}$ holds. By the comments after \eqref{2.eqn:def_of_A-sub-k}, this proves the lemma.
		
		It remains to consider the case when \eqref{2.eqn:case_E-sub-k+1_small} is false and so $E_{k+1} \geq C^{-1}E_k$ holds. At the start of the inductive step we supposed that~\ref{2.itm:union_of_cubes}$_{k-1}$ holds, so the set $K_{k-1}\brax{E_1,\dotsc,E_k}$ may be covered by $O_{C,n}(B^{\sigma}(E_1\dotsm E_k)E_k^{-\sigma-k})$ boxes of side $E_k$. We have
		\[
		K_{k}\brax{E_1,\dotsc,E_{k+1}}\subset K_{k-1}\brax{E_1,\dotsc,E_k},
		\]
		and so the set $K_{k}\brax{E_1,\dotsc,E_{k+1}}$ is also covered by this collection of boxes.  Since  $E_{k+1} \geq C^{-1}E_k$ holds, we can divide each of these boxes into $O_{C,n}(1)$ boxes of side $E_{k+1}$.  This proves~\ref{2.itm:union_of_cubes}$_{k}$.
\end{proof}

\section{Small values of a trilinear form}\label{2.sec:finding_spaces_X_and_Y}

Part~\ref{2.itm:finding_spaces_X_and_Y} of Davenport's argument from \S\ref{2.sec:outline} starts from a point $\vec{x}$ for which the matrices $\cubicMatrix$ and $\cubicMinorsJacobian{ b+1 }$ have prescribed ranks, and finds linear spaces $X, Y$ such that for all $\vec{X}\in X$ and $\vec{Y},\vec{Y}' \in Y$ the equation  $\vec{Y}^T H_c(\vec{X}) \vec{Y}' = 0$ holds.  Our analogue is the following pair of results, which give linear spaces on which the trilinear form $\vec{Y}^T H_c(\vec{X}) \vec{Y}'$ is small.


\begin{lemma}\label{2.lem:davenport's_trick}
	
	Let $c(\vec{x})$ be as in Definition~\ref{2.def:aux_ineq}, and let $\cubicEigenvalue{ i }$ and $\cubicMinorsJacobian{ i }$ be as in Definition~\ref{2.def:cubic_eigenvalues}. Suppose that $b \in \set{ 1,\dotsc,n-1 }$ and that $\vecsuper{x}{0}\in \bbR^n$. Then, provided  $\cubicMinorsAt{b}{\vecsuper{x}{0}}$ is nonzero, there exists an $(n-b)$-dimensional linear subspace $Y$ of $ \bbR^n$ such that for all $ \vec{Y},\vec{Y}'\in Y$ and all $\vec{t}\in\bbR^n$ we have
	\begin{equation}\label{2.eqn:davenport's_trick}
	\vec{Y}^T\cubicMatrixAt{\vec{t}}\vec{Y}'
	\ll_n
	\brax[\bigg]{
		\frac{ \supnorm{ \cubicMinorsJacobianAt{ b+1 }{ \vecsuper{x}{0} } \vec{t} } }{ \supnorm{ \cubicMinorsAt{ b }{ \vecsuper{x}{0} } } }
		+
		\frac{ \abs{ \cubicEigenvalueAt{ b+1 }{ \vecsuper{x}{0} } } \cdot\supnorm{ \vec{t} } }{ \abs{ \cubicEigenvalueAt{ b }{ \vecsuper{x}{0} } } }}
	\supnorm{\vec{Y}}\supnorm{\vec{Y}'}.
	\end{equation}
\end{lemma}

By setting  $b=\rank\cubicMatrixAt{\vecsuper{x}{0}}$ and $X=\ker \cubicMinorsJacobianAt{b+1}{\vecsuper{x}{0}},$ one may recover Davenport's result. We prove Lemma~\ref{2.lem:davenport's_trick} at the end of this section, after deducing


\begin{corollary}\label{2.cor:applying_davenport's_trick}	
	Let $c,$ $\cubicMatrix$ and  $\auxIneqOfCNumSolns(B)$ be as in Definition~\ref{2.def:aux_ineq}. For any $B,C \geq 1$ and any $\sigma\in \set{ 0,\dotsc,n-1 },$ one of the following alternatives holds. Either
	\begin{equation}\label{2.eqn:bound_on_aux_ineq}
	\auxIneqOfCNumSolns(B) \ll_{C,n} B^{n+\sigma}\brax{\log B}^n,
	\end{equation}
	or there exist positive-dimensional linear subspaces $X$ and $Y$ of $ \bbR^n$ for which $\affdim X +\affdim Y = n+\sigma+1,$ such that
	\begin{align}\label{2.eqn:spaces_X,Y_on_which_H_small}
	\abs{\vec{Y}^{T}\cubicMatrixAt{\vec{X}}\vec{Y}'}
	&\ll_n
	C^{-1}
	\supnorm{\vec{Y}}\supnorm{\vec{X}}\supnorm{\vec{Y}'}
	&\text{for all }\vec{X}\in X, \,\vec{Y},\vec{Y}'\in Y.
	\end{align}
\end{corollary}

\begin{proof}
	We saw in Lemma~\ref{2.lem:union_of_cubes} that for any $k\in\set{0,\dotsc,n-\sigma-1}$ and any $E_1,\dotsc,E_{k+1}\in\bbR$ satisfying
	\[
	CB \geq E_1 \geq \dotsb \geq E_{k+1}\geq 1,
	\]
	one of~\ref{2.itm:union_of_cubes}$_{k}$,~\ref{2.itm:some_Jacobian_close_to_having_small_rank}$_{k}$, or~\ref{2.itm:c_close_to_a_cone} must hold. Suppose first that in every case alternative~\ref{2.itm:union_of_cubes}$_{k}$ holds. By \eqref{2.eqn:integer_points_in_E-cubes}, we then have
	\begin{equation}\label{2.eqn:points_in_cubes_aux_ineq_proof}
	\#\setbrax{\bbZ^n\cap K_{k}\brax{E_1,\dotsc,E_{k+1}}}
	\ll_{C,n}
	B^{\sigma} (E_1 \dotsm E_{k+1}) E_{k+1}^{n-\sigma-k-1}
	\end{equation}
	for every $k\in\set{0,\dotsc,n-\sigma-1}$ and every $CB \geq E_1 \geq \dotsb \geq E_{k+1}\geq 1$. Now
	Corollary~\ref{2.cor:many_x_with_prescribed_eigenvalues} shows that either
	\begin{equation}\label{2.eqn:H_small_aux_ineq_proof}
	\frac{\auxIneqOfCNumSolns\brax{B}}{B^n \brax{\log B}^n} \ll_n \#\setbrax{\bbZ^n \cap K_0(1)},
	\end{equation}
	or
	\begin{equation}\label{2.eqn:prescribed_eigenvalues_aux_ineq_proof}
	\frac{ 2^{e_1+\dotsb+e_k} \auxIneqOfCNumSolns\brax{B}}{B^n \brax{\log B}^n}
	\ll_n
	\#\setbrax[\big]{ \bbZ^n \cap K_k\brax{2^{e_1}, \dotsc, 2^{e_k}, 1} },
	\end{equation}
	where $k\in \set{1,\dotsc,n-1}$ and the inequalities $B \gg_n 2^{e_1} \geq \dotsb \geq 2^{e_{k+1}}\geq 1$ hold, or
	\begin{equation}\label{2.eqn:all_eigenvalues_prescribed_aux_ineq_proof}
	\frac{ 2^{e_1+\dotsb+e_n} \auxIneqOfCNumSolns\brax{B}}{B^n \brax{\log B}^n}
	\ll_n
	\#\setbrax[\big]{ \bbZ^n \cap K_{n-1}\brax{2^{e_1}, \dotsc, 2^{e_n}} }
	\end{equation}
	where the inequalities $B \gg_n 2^{e_1} \geq \dotsb \geq 2^{e_{n}}\geq 1$ hold. We may assume that $C$ is sufficiently large in terms of $n,$ so that $CB \geq 2^{e_1}$ holds in \eqref{2.eqn:prescribed_eigenvalues_aux_ineq_proof}-\eqref{2.eqn:all_eigenvalues_prescribed_aux_ineq_proof}. Substituting  the bound \eqref{2.eqn:points_in_cubes_aux_ineq_proof} into each of \eqref{2.eqn:H_small_aux_ineq_proof}-\eqref{2.eqn:all_eigenvalues_prescribed_aux_ineq_proof} proves the conclusion \eqref{2.eqn:bound_on_aux_ineq}.
	
	Suppose next that alternative~\ref{2.itm:c_close_to_a_cone} holds in Lemma~\ref{2.lem:union_of_cubes}. In this case we let $Y=\bbR^n,$ and the conclusion \eqref{2.eqn:spaces_X,Y_on_which_H_small} follows from \eqref{2.eqn:c_close_to_a_cone_at_many_points}.
	
	It remains to treat the case when there exist $k\in\set{0,\dotsc,n-\sigma-1}$ and $CB \geq E_1 \geq \dotsb \geq E_{k+1}\geq 1$ such that alternative~\ref{2.itm:some_Jacobian_close_to_having_small_rank}$_{k}$ holds in  Lemma~\ref{2.lem:union_of_cubes}. This means that there exist an integer $1 \leq b \leq k,$ a point $\vecsuper{x}{0}\in K_{b}\brax{E_1,\dotsb,E_{b+1}},$ and a $(\sigma+b+1)$-dimensional linear subspace $X$ of $ \bbR^n$ such that
	\begin{align}
	\label{2.eqn:ith_minors_small_application}
	E_{b+1}
	&<
	C^{-1}E_{b},
	\intertext{and}
	\label{2.eqn:jacobian_small_application}
	\supnorm{\cubicMinorsJacobianAt{ b+1 }{ \vecsuper{x}{0} } \vec{X}}
	&\leq
	C^{-1}
	\supnorm{ \cubicMinorsAt{ b }{ \vecsuper{x}{0} } }\supnorm{ \vec{X} }
	&\text{for all } \vec{X}\in X.
	\end{align}
	
	Since $\vecsuper{x}{0}\in K_k\brax{ E_1,\dotsc, E_{k+1}},$ the inequalities $\frac{1}{2}E_i <  \cubicEigenvalueAt{ i }{ \vecsuper{ x }{ 0 } }\leq E_i$ hold. Therefore \eqref{2.eqn:ith_minors_small_application} implies
	\begin{equation}
	\cubicEigenvalueAt{b+1}{\vecsuper{x}{0}}
	<
	2C^{-1}
	\cubicEigenvalueAt{b}{\vecsuper{x}{0}},
	\label{2.eqn:lambda-sub-i_application}
	\end{equation}
	Note that \eqref{2.eqn:lambda-sub-i_application} implies that $\cubicEigenvalueAt{b}{\vecsuper{x}{0}}$ is nonzero, and so $\cubicMinorsAt{b}{\vecsuper{x}{0}}$ is nonzero, by Lemma~\ref{2.lem:minors_and_eigenvalues}. Hence we may apply Lemma~\ref{2.lem:davenport's_trick} to find an $(n-b)$-dimensional linear space $Y$ such that for all $\vec{Y},\vec{Y}'\in Y$ and all $\vec{t}\in\bbR^n$ the bound \eqref{2.eqn:davenport's_trick} holds. The conclusion \eqref{2.eqn:spaces_X,Y_on_which_H_small} follows on taking $\vec{t}=\vec{X}$ in  \eqref{2.eqn:davenport's_trick} of Lemma~\ref{2.lem:davenport's_trick} and substituting in the bounds \eqref{2.eqn:jacobian_small_application} and \eqref{2.eqn:lambda-sub-i_application}.
\end{proof}

\begin{proof}[Proof of Lemma~\ref{2.lem:davenport's_trick}]
	We imitate the proof of Lemma~3 in Davenport~\cite{davenportSixteen}, which begins by considering the following easy ``warm-up" problem. Suppose we were to look for $n-b$ linearly independent vectors $\vec{y}$ at which $\cubicMatrixAt{ \vecsuper{ x }{ 0 } }\vec{y}$ vanishes. One approach would be as follows. One can construct matrices $L^{(i)},$ $M^{(i)}$ for $i=1,\dotsc,n-b,$ with entries in $\set{ 0,\pm 1 },$ such that the vectors
	\begin{align}
	\label{2.eqn:def_of_y}
	\vecsuper{ y }{ i }\brax{ \vec{x} }
	&\eqdef
	L^{(i)}
	\cubicMinors{ b }
	\intertext{satisfy}
	\label{2.eqn:H(x)y}
	\cubicMatrix\vecsuper{y}{i}\brax{\vec{x}}
	&=
	M^{(i)}
	\cubicMinors{b+1}.
	\end{align}
	That is, the components of $\vecsuper{y}{i}(\vec{x})$ are polynomials of the form $\pm\Delta_j^{(b)}(\vec{x}),$ and the components of $\cubicMatrix\vecsuper{y}{i}(\vec{x})$ are polynomials of the form $\pm\Delta_j^{(b+1)}(\vec{x})$.
	
	If $\cubicMatrixAt{ \vecsuper{ x }{ 0 } }\vec{y}=\vec{0}$ had exactly $n-b$ linearly independent solutions $\vec{y},$ we would have $\cubicMinorsAt{b+1}{\vecsuper{x}{0}}=\vec{0},$ while  $\cubicMinorsAt{b}{\vecsuper{x}{0}}$ would be nonzero. We would then have $n-b$ solutions $\vecsuper{Y}{k}$ defined by
	\begin{align}\label{2.eqn:def_of_big_Y}
	\vecsuper{Y}{k}
	&\eqdef
	\frac{
		\vecsuper{ y }{ k }\brax{ \vecsuper{x}{0} }
	}{
		\supnorm{\cubicMinorsAt{ b }{\vecsuper{x}{0}}}
	}
	&(1 \leq k\leq n-b),
	\end{align}
and if we chose our matrices $L^{(i)}, M^{(i)}$ appropriately these would be linearly independent.
	
	We now return to the proof of the lemma. Assume for the time being that $L^{(i)},$ $M^{(i)},$ $\vecsuper{y}{i}(\vec{x})$ and $\vecsuper{Y}{i}$ satisfying \eqref{2.eqn:def_of_y}-\eqref{2.eqn:def_of_big_Y} are given, and let $\vecsuper{x}{0}$ be as in the lemma. Let $\vec{t}\in\bbR^n$. Let $\partial_{\vec{t}}$ be the directional derivative along $\vec{t}$ defined by  $\sum t_i \frac{\partial}{\partial x_i},$ and apply $\partial_{\vec{t}}$ to both sides of \eqref{2.eqn:H(x)y}. This shows that
	\begin{equation}\label{2.eqn:davenport's_trick_I}
	\sqbrax[\big]{\partial_{\vec{t}}
		\cubicMatrix}
	\vecsuper{ y }{ i }\brax{ \vec{x} }
	+
	\cubicMatrix \sqbrax[\big]{\partial_{\vec{t}} \vecsuper{ y }{ i }\brax{ \vec{x} }}
	=
	M^{(i)}
	\sqbrax[\big]{\partial_{\vec{t}}\cubicMinors{ b+1 }}.
	\end{equation}
	Since  $\partial_{\vec{t}}$ is a linear operator and we have
	\begin{equation*}
	\partial_{\vec{t}}  \cubicMinors{ k }
	=
	\cubicMinorsJacobian{ k }\vec{t},
	\end{equation*}
	 it follows from \eqref{2.eqn:davenport's_trick_I} that
	\begin{align}
	\cubicMatrixAt{ \vec{t} }
	\vecsuper{ y }{ i }\brax{ \vec{x} }
	={}&
	M^{(i)}
	\cubicMinorsJacobian{ b+1 }\vec{t}
	-\cubicMatrix
	L^{(i)} \partial_{\vec{t}} \cubicMinors{ b }.
	\nonumber
	\intertext{Premultiplying by $\vecsuper{ y }{ j }\brax{ \vec{x} }^T$ and using \eqref{2.eqn:H(x)y} gives}
	\vecsuper{ y }{ j }\brax{ \vec{x} }^T
	\cubicMatrixAt{ \vec{t} }
	\vecsuper{ y }{ i }\brax{ \vec{x} }
	={}&	
	\vecsuper{ y }{ j }\brax{ \vec{x} }^T
	M^{(i)}
	\cubicMinorsJacobian{ b+1 }\vec{t}
	\nonumber
	\\
	&-
	\sqbrax[\big]{M^{(i)}\cubicMinors{ b+1 }}^T 
	\sqbrax[\big]{ L^{(i)} \partial_{\vec{t}} \cubicMinors{ b } }.
	\label{2.eqn:davenport's_trick_II}
	\end{align}
	Now Lemma~\ref{2.lem:minors_and_eigenvalues} shows that
	\begin{equation*}
	\frac{\supnorm{ \cubicMinors{ b+1 } }}{\supnorm{ \cubicMinors{ b } }}
	\ll_n \abs{\cubicEigenvalue{ b+1 }},
	\qquad
	\frac{\supnorm{ \partial_{\vec{t}} \cubicMinors{ b } }}{\supnorm{ \cubicMinors{ b } }}
	\ll_n \frac{\supnorm{ \vec{t} }}{\abs{\cubicEigenvalue{ b}}}
	\end{equation*}
	and substituting these bounds into \eqref{2.eqn:davenport's_trick_II} gives
	\begin{equation}
	\vec{Y}^{(j)T}
	\cubicMatrixAt{ \vec{t} }
	\vecsuper{Y}{i}
	\ll_n
	\frac{
		\supnorm{\cubicMinorsJacobianAt{ b+1 }{\vecsuper{x}{0}}\vec{t}}
	}{ \supnorm{ \cubicMinorsAt{ b }{\vecsuper{x}{0}} } }
	+
	\frac{ \abs{\cubicEigenvalueAt{ b+1 }{\vecsuper{x}{0}}} \cdot \supnorm{\vec{t}} }{ \abs{\cubicEigenvalueAt{ b }{\vecsuper{x}{0}}} }
	\label{2.eqn:davenport's_trick_with_y_not_nec_orthonormal}
	\end{equation}
	where the $\vecsuper{Y}{k}$ are as in \eqref{2.eqn:def_of_big_Y}.
	
	The idea is now to let $Y$ be the span of the  $\vecsuper{Y}{k}$ and deduce \eqref{2.eqn:davenport's_trick} from \eqref{2.eqn:davenport's_trick_with_y_not_nec_orthonormal}. Since we are looking for an $(n-b)$-dimensional space $Y$ we will need $\vecsuper{Y}{1},\dotsc,\vecsuper{Y}{n-b}$ to be linearly independent. In order to prove \eqref{2.eqn:davenport's_trick} we require the following slightly stronger statement. We claim there are $L^{(i)},$ $M^{(i)},$ $\vecsuper{y}{i}\brax{\vec{x}}$ and $\vecsuper{Y}{i}$ satisfying \eqref{2.eqn:def_of_y}-\eqref{2.eqn:def_of_big_Y}, such that  the linear combination defined by $
	\vec{Y} = \sum_{i=1}^{n-b} \gamma_i \vecsuper{ Y }{ i }
	$ satisfies $
	\supnorm{\vec{\gamma}}\ll_n \supnorm{\vec{Y}}
	$ for every vector $
	\vec{\gamma} $ in real $(n-b)$-space. The lemma then follows, with $Y$ being the span of the $\vecsuper{ Y }{ i },$ on expressing $\vec{Y},\vec{Y}'$ as linear combinations of the $\vecsuper{Y}{i}$ and applying \eqref{2.eqn:davenport's_trick_with_y_not_nec_orthonormal}.
	
	For the remainder of the proof we will assume for simplicity that the $b\times b$ minor of $\cubicMatrixAt{\vecsuper{x}{0}}$ with largest absolute value is the minor in the lower right-hand corner, that is, we will assume that
	\begin{equation}\label{2.eqn:largest_minor_in_bottom_right}
	\supnorm{\cubicMinorsAt{b}{\vecsuper{x}{0}}}
	=
	\abs[\big]{ \det \brax[\big]{ \brax{\cubicMatrixAt{\vecsuper{x}{0}}_{k\ell}}_{\substack{k = n-b+1,\dotsc, n \\ \ell = n-b+1,\dotsc,n }} } }.
	\end{equation}
	In general \eqref{2.eqn:largest_minor_in_bottom_right} holds after permuting the rows and columns of the matrix $\cubicMatrix$  and one can then apply the same permutations throughout the rest of our construction of $\vecsuper{Y}{i},$ every time the matrix $\cubicMatrix$ appears.
	
	Define $\vecsuper{y}{1}\brax{\vec{x}},\dotsc,\vecsuper{y}{n-b}\brax{\vec{x}}$ by
	\begin{equation*}
	y^{ (i) }_j\brax{\vec{x}}
	\eqdef
	\begin{cases}
	(-1)^{n-b}\det \brax[\big]{ \brax{\cubicMatrix_{k\ell}}_{\substack{k = n-b+1,\dotsc, n \\ \ell = n-b+1,\dotsc,n }} }
	&\text{if } j=i,
	\\
	(-1)^{j}\det \brax[\big]{ \brax{\cubicMatrix_{k\ell}}_{\substack{k = n-b+1,\dotsc, n \\ \ell = i, n-b+1,\dotsc,n;\; \ell \neq j }} }
	&\text{if } j >n-b,
	\\
	0 &\text{otherwise}
	\end{cases}
	\end{equation*}
	where $(\ell =i, n-b+1,\dotsc,n;\; \ell \neq j )$ means that $\ell$ first takes the value $i$ and then runs over the numbers $n-b+1,\dotsc,n$ with $j$ omitted. Now this is of the form \eqref{2.eqn:def_of_y}, and one can check that
	\begin{equation*}
	(\cubicMatrix\vecsuper{ y }{ i }\brax{\vec{x}})_j
	=
	\begin{cases}
	(-1)^{n-b}\det \brax[\big]{ \brax{\cubicMatrix_{k\ell} }_{\substack{k = j,n-b+1,\dotsc, n \\ \ell = i,n-b+1,\dotsc,n }} }
	&\text{if } j \leq n-b
	\\
	0
	&\text{otherwise}
	\end{cases}		
	\end{equation*}
	which is of the form \eqref{2.eqn:H(x)y}.
	Define a matrix $Q$ by
	\[
	Q\eqdef
	\left(
	\begin{array}{@{}c|c|c|c|c|c@{}}
	\vecsuper{Y}{1}
	&\cdots 
	&\vecsuper{Y}{n-b}
	&\vecsuper{e}{n-b+1}
	&\cdots
	&\vecsuper{e}{n}
	\end{array}
	\right),
	\]
	or equivalently by
	\[
	Q
	= 	
	\left(
	\begin{array}{@{}c|c|c|c|c|c@{}}
	\frac{\vecsuper{y}{1}(\vecsuper{x}{0})}{\supnorm{\cubicMinorsAt{b}{\vecsuper{x}{0}}}}
	&\cdots
	&\frac{\vecsuper{y}{n-b}(\vecsuper{x}{0})}{\supnorm{\cubicMinorsAt{b}{\vecsuper{x}{0}}}}
	& \vecsuper{e}{n-b+1}
	& \cdots 
	&\vecsuper{e}{n}
	\end{array}
	\right),
	\]
so that the entries $Q_{ij}$ have absolute value at most 1. Then one sees from \eqref{2.eqn:largest_minor_in_bottom_right} that
	\begin{equation*}
	Q
	=
	\begin{pmatrix}
	I_{n-b} & 0_{b\times b}
	\\
	\widetilde{Q} & I_b 
	\end{pmatrix},
	\end{equation*}
	where $\widetilde{Q}$ is some $(n-b)\times (n-b)$ matrix. In particular $\det Q = 1,$ and so the entries of ${Q^{-1}}$ are bounded in terms of $n$. It follows that if $\vec{Y} = \sum_{i=1}^{n-b} \gamma_i \vecsuper{ Y }{ i }$ then $\gamma_i = \brax{Q^{-1}\vec{Y}}_i \ll_n \supnorm{\vec{Y}},$ as claimed.
\end{proof}

\section{Constructing singular points on $V(c)$}\label{2.sec:cubic_aux_ineq_prop}

Corollary~\ref{2.cor:applying_davenport's_trick} shows that either $\auxIneqOfCNumSolns\brax{ B }$ is small, or there are spaces $X,Y$ of large dimension on which $\vec{Y}^{T}H_{c}\brax{\vec{X}}\vec{Y}'$ is small. To prove Proposition~\ref{2.prop:cubic_aux_ineq} we show that the second alternative implies that $V(c)$ is singular. This is our analogue of Davenport's step~\ref{2.itm:contructing_solutions_with_spaces_X_and_Y}, as described in \S\ref{2.sec:outline}.

\begin{proof}[Proof of Proposition~\ref{2.prop:cubic_aux_ineq}]
	Suppose for a contradiction that the result is false. Then for every $N\in \bbN$ there is $c_N \in \calK$ with
	\begin{equation*}
	\auxIneqOfCNumSolns\brax{ B } \geq N B^{n+\sigma_\calK}\brax{\log B}^n.
	\end{equation*}
	By Corollary~\ref{2.cor:applying_davenport's_trick}, this implies that there are linear subspaces $X_N,Y_N$ of $\bbR^n$ such that
	\[
	\affdim X_N+\affdim Y_N = n+\sigma_\calK+1
	\]
	holds and for all $\vec{X}\in X_N$ and $\vec{Y},\vec{Y}'\in Y_N,$ we have
	\begin{equation*}
	\abs{\vec{Y}^{T}H_{c_N}\brax{\vec{X}}\vec{Y}'}
	\leq
	N^{-1}
	\supnorm{\vec{Y}}\supnorm{\vec{X}}\supnorm{\vec{Y}'}.
	\end{equation*}
	If we multiply $c_N$ by a constant then the matrix $H_{c_N}\brax{\vec{x}}$ does not change. So we may assume that for each $N$ the equality $\supnorm{c_N}=1$ holds. After passing to a subsequence we have $c_N\to c$ as $N\to\infty,$ and it follows that there are subspaces  $X,Y$ of $\bbR^n$ such that $\affdim X+\affdim Y = n+\sigma_\calK+1$ and	
%
	\begin{align}
	\vec{Y}^{T}H_{c}\brax{\vec{X}}\vec{Y}'
	&=0
	&\text{for all }
	\vec{X}\in X,\,\vec{Y},\vec{Y}'\in Y.
	\label{2.eqn:X,_Y_on_which_c_vanishes}
	\end{align}
Let $b \in \set{0,\dotsc,n-\sigma-1}$ such that
	\begin{align*}
	\affdim X &=n-b,
	&
	\affdim Y &=\sigma_\calK+b+1.
	\end{align*}
	Let $\vecsuper{x}{1},\dotsc, \vecsuper{x}{n}$ be a basis of $\bbR^n$ such that $\vecsuper{x}{b+1},\dotsc,\vecsuper{x}{n}$ is a basis of $X$.
	
	Let $\sqbrax{Y}$ be the projective linear space in $\bbP_{\bbR}^{n-1}$ associated to $Y $. Take homogeneous co-ordinates $\vec{y}$ on $\sqbrax{Y},$ so that $\vec{y}$ takes values in $Y$.
	
	Let $W$ be the projective variety cut out in $\sqbrax{Y}$ by the $b$ equations
	\begin{align}
	W:\vec{y}^T \cubicMatrixAt{\vecsuper{x}{i}} \vec{y}
	&= 0
	&\brax{i= 1,\dotsc, b },
	\label{2.eqn:constructing_points_on_sing_V}
	\end{align}
	so that	
	\[
	\projdim W \geq \projdim \sqbrax{Y} -b = \sigma_\calK.
	\]
	We claim that $W$ is contained in the singular locus of the projective hypersurface $V(c)$. It follows that $\projdim\sing V(c) \geq \sigma_\calK,$ which is a contradiction, by \eqref{2.eqn:def_of_sigma-sub-calK}.
	
	Now \eqref{2.eqn:X,_Y_on_which_c_vanishes} implies that for every $\vec{y}\in Y$ we have
	\begin{align*}
	\vec{y}^T \cubicMatrixAt{\vecsuper{x}{i}} \vec{y}
	&=0
	&(i=b+1,\dotsc,n).
	\end{align*}
	So if we let $\vec{y}\in Y$ such that \eqref{2.eqn:constructing_points_on_sing_V} holds, then we have
	\begin{align*}
	\vec{y}^T \cubicMatrix \vec{y}
	&=0
	&\text{for all }\vec{x}\in \bbR^n.
	\end{align*}
	This implies that $\grad_{\vec{y}}c(\vec{y}) = \vec{0}$ holds, by the definition \eqref{2.eqn:def_of_H}.  It follows that every point of $W$ is contained in $\sing V(c),$ as claimed.
\end{proof}

\section{Acknowledgements}

This work is based on a DPhil thesis submitted to Oxford University, and was supported by EPSRC grants EP/J500495/1 and EP/M507970/1. I would like to thank my DPhil supervisor, Roger Heath-Brown. I am grateful to Rainer Dietmann for helpful conversations.

\bibliography{systems-of-many-forms}

\end{document}